\documentclass[a4paper,10pt]{article}

\usepackage{geometry}
\usepackage{amsmath}
\usepackage{amsfonts,color,amssymb,mathrsfs,stmaryrd}
\usepackage{amsthm}
\usepackage{url}

\usepackage{bera}

\usepackage[normalem]{ulem}

\usepackage{graphicx}
\usepackage{wrapfig}
\usepackage{mathrsfs}
\usepackage{bbm}
\usepackage{enumitem}
\usepackage{tikz}
\usepackage{hyperref}
\usetikzlibrary{calc}
\usetikzlibrary{intersections}
\usetikzlibrary{angles}
\usetikzlibrary{quotes}
\usetikzlibrary{babel}

\newcommand{\XX}{\mathbb{X}}

\newcommand{\p}{\mathbb{P}}
\newcommand{\E}{\mathbb{E}}

\renewcommand{\H}{{\cal P}}

\newcommand{\cB}{\mathcal{B}}
\newcommand{\eps}{\varepsilon}

\numberwithin{equation}{section}

\newtheorem{theorem}{Theorem}\numberwithin{theorem}{section}

\newtheorem{prop}[theorem]{Proposition}
\newtheorem{lemma}[theorem]{Lemma}

\newtheorem{rema}[theorem]{Remark}

\def\N{\mathbb{N}}

\def\E{\mathbb{E}}

\def\H{\mathbb{H}}
\def\VV{\mathbb{V}}
\def\0{{\bf 0}}

\def\R{\mathbb{R}}
\def\PP{\mathbb{P}}

\newcommand{\HB}{\mathrm{HB}}

\renewcommand{\E}{\mathbb E \,}

\newcommand{\Var}{{\rm Var}}
\newcommand{\x}{{\bf x}}

\def\beqn{\begin{equation}}
	\def\eeqn{\end{equation}}
\def\be{\begin{equation}}
	\def\ee{\end{equation}}

\def\R{\mathbb{R}}

\def\dint{\textup{d}}

\def\cH{{\cal H}}
\def\cL{{\cal L}}

\def\qed{\hfill\hbox{${\vcenter{\vbox{
					\hrule height 0.4pt\hbox{\vrule width 0.4pt height 6pt
						\kern5pt\vrule width 0.4pt}\hrule height 0.4pt}}}$}}

\newcommand{\blue}[1]{\textcolor{blue}{#1}}

\def\dint{\mathrm{d}}
\DeclareMathOperator{\arcosh}{arcosh}
\DeclareMathOperator{\RST}{ \mathrm{RST} }
\DeclareMathOperator{\DSF}{ \mathrm{DSF} }
\newcommand{\kk}{ \varkappa }
\newcommand{\skk}{\sqrt{\kk}}
\newcommand{\id}{\mathbbm{1}}

\renewcommand{\thefootnote}{\fnsymbol{footnote}}

\usepackage{titlesec}

\titleformat*{\section}{\normalfont\large\bfseries}
\titleformat*{\subsection}{\normalfont\bfseries}

\date{\vspace{-0.95cm}}

\DeclareMathOperator*{\argmin}{argmin}

\makeatletter
\let\@fnsymbol\@alph
\makeatother

\begin{document}

\title{\bfseries The radial spanning tree in hyperbolic space}

\author{%
	\hspace*{1.6cm} 	  Daniel Rosen\footnotemark[1]
	\and  Matthias Schulte\footnotemark[2]\hspace*{1.6cm} 
	\and Christoph Th\"ale\footnotemark[3]%
	\and Vanessa Trapp\footnotemark[4]%
}

\date{}
\renewcommand{\thefootnote}{\fnsymbol{footnote}}

\footnotetext[1]{%
	Technical University of Dortmund, Germany. Email: daniel.rosen@math.tu-dortmund.de
}	

\footnotetext[2]{%
Hamburg University
of Technology, Germany. Email: matthias.schulte@tuhh.de
}

\footnotetext[3]{%
	Ruhr University Bochum, Germany. Email: christoph.thaele@rub.de
}

\footnotetext[4]{%
	Hamburg University
	of Technology, Germany. Email: vanessa.trapp@tuhh.de
}

\maketitle

\begin{abstract}\noindent
Consider a stationary Poisson process $\eta$ in a $d$-dimensional hyperbolic space of constant curvature $-\varkappa$ and let the points of $\eta$ together with a fixed origin $o$ be the vertices of a graph. Connect each point $x\in\eta$ with its radial nearest neighbour, which is the hyperbolically nearest vertex to $x$ that is closer to $o$ than $x$. This construction gives rise to the hyperbolic radial spanning tree, whose geometric properties are in the focus of this paper. In particular, the degree of the origin is studied. For increasing balls around $o$ as observation windows, expectation and variance asymptotics as well as a quantitative central limit theorem for a class of edge-length functionals are derived. The results are contrasted with those for the Euclidean radial spanning tree.
	
	
\end{abstract}

\section{Introduction and main results}

\subsection{General introduction}

Stochastic geometry is the branch of probability theory concerned with spatial random structures such as random graphs or random tessellations. While the underlying space is classically the Euclidean space, a recent trend is to investigate such models also in non-Euclidean spaces and, in particular, in the hyperbolic space. Following this line of research, we consider the so-called radial spanning tree, a well-known model in the Euclidean case, in hyperbolic space.

Studying models from stochastic geometry in hyperbolic space of constant negative curvature has several motivations. From a mathematical point of view replacing Euclidean space by other underlying spaces is a natural problem and the hyperbolic space is the first candidate. Working with such a space is helpful in understanding which properties of a stochastic geometry model depend on the curvature and which do not. On the other hand, many real-world networks such as the internet graph seem to have an underlying hyperbolic structure \cite{BogunaEtAl,KrioukovEtAl}. In fact, the random geometric graph in hyperbolic space introduced in these papers shares many crucial properties of complex networks.

When replacing Euclidean by hyperbolic space some new phenomena can emerge as illustrated by the following non-exhaustive list of examples. In \cite{BS}, it was shown that for some parameter choices in Poisson-Voronoi-Bernoulli percolation in the hyperbolic plane one simultaneously has infinitely many black and infinitely many white components, which cannot occur in Euclidean space. For continuum percolation of the Boolean model in hyperbolic space it was shown in \cite{Tykesson} that there exists a regime with infinitely many unbounded components, a situation not possible in Euclidean space. For the low intensity Boolean model in the hyperbolic plane it was shown in \cite{BJST} that an observer located at a fixed non-covered point can see infinitely far in a fractal set of directions, a phenomenon in sharp contrast to the Euclidean situation, where such an observer can only see a bounded region. The same has also been shown in \cite{BJST}  for Poisson line processes in the hyperbolic plane. While the surface area of a Poisson hyperplane process in an increasing observation window satisfies a central limit theorem in Euclidean space, this is only true in dimensions two and three in hyperbolic space \cite{HHT}, whereas non-Gaussian limiting distributions arise for all higher space dimensions \cite{KRT}.

In this paper, we consider the radial spanning tree, which was introduced as a random graph in Euclidean space in \cite{BB}, motivated by applications in communication networks. Its vertices are the points of a stationary Poisson process together with the origin. From each vertex except of the origin an edge is drawn to its nearest neighbour that is closer to the origin. The resulting graph is a random tree rooted at the origin. In \cite{BB} and \cite{BCT}, semi-infinite paths of the radial spanning tree in $\mathbb{R}^2$ were studied. Variance asymptotics and quantitative central limit theorems for edge-length functionals were derived in \cite{ST}. A crucial role for the analysis of the radial spanning tree plays the directed spanning forest, where each point is connected with its closest neighbour in the half-space generated by a fixed direction.

In hyperbolic space, the radial spanning tree was first considered in \cite{CFT1}, where the same path properties as for the two-dimensional Euclidean radial spanning tree were shown for any dimension. These results were further refined in \cite{CFT2}. The hyperbolic directed spanning forest was first studied in \cite{Flammant}. It shows a radically different behaviour than its Euclidean analogue. It is a tree containing infinitely many bi-infinite branches, while the Euclidean radial spanning forest is a countable collection of trees for $d\geq 4$ and does not admit bi-infinite branches.

In this paper, we investigate for the radial spanning tree in hyperbolic space the degree of the origin and sums of powers of the edge-lengths within an observation window. For the degree of the origin we show that it is in contrast to the Euclidean case not bounded. For the edge-length functionals evaluated for balls of increasing radius around the origin, we derive asymptotic formulas for the expectation and variance and establish a quantitative central limit theorem. These results are similar to those derived in \cite{ST} for the Euclidean case. The major difference is that a hyperbolic ball scales differently with its radius than a Euclidean ball and that in the asymptotic variance formula the integration over a half-space is replaced by integration over a horoball. Horoballs are important objects of hyperbolic geometry without Euclidean counterparts. They are obtained as limits of increasing balls with a fixed point in their boundaries and centres tending to infinity along a geodesic ray.

\medspace

This paper is organised as follows. In Subsection \ref{sec:MainResults}, we introduce the radial spanning tree constructed from the points of a stationary Poisson process in hyperbolic space and present our main results on the degree of the origin and the asymptotic behaviour of the edge-length functionals. They are compared with those for the Euclidean case in Subsection \ref{subsec:ComparisonEuclidean}. In Section \ref{sec:background_hyperbolic}, we collect important facts about hyperbolic space, especially some integration formulas, while Section \ref{sec:PPP} is devoted to Poisson processes and provides results for functionals of Poisson processes that are employed throughout the paper. Our findings for the degree of the origin are shown in Section \ref{sec:proofs_origin}. For the edge-length functionals the asymptotic behaviour of the expectation is established in Section \ref{sec:exp_LR}. The proofs of the variance asymptotics and the central limit theorem are presented in Sections \ref{sec:variance} and \ref{sec:CLT} and rely on a scaling property for the edge-length functionals considered in Section \ref{sec:scaling}.

\subsection{Model and main results}\label{sec:MainResults}

We start this section with a detailed description of our model; for unexplained notions and notation from hyperbolic geometry we refer to Section \ref{sec:background_hyperbolic} and for background material on {Poisson processes} to Section \ref{sec:PPP}.

 First, let $(X,d)$ be a metric space with a chosen base point $o\in X$, and let $\xi \subseteq X$ be a locally finite subset. The radial spanning tree $\RST(\xi)$ on $\xi$ with respect to the base point $o$ is the tree rooted at $o$ and defined as follows. The vertices of $\RST(\xi)$ are the points of $\xi\cup\{o\}$. Any vertex {$x \in \xi \setminus\{o\}$} is connected by an edge to its \emph{radial nearest neighbour} $n(x,\xi)$, which is by definition the nearest vertex to $x$ among all vertices which lie in the {open} ball $B(o, d(o, x))$ centred at $o$ with radius $d(o,x)$, that is,
\[
n(x, \xi) := \argmin_{\substack{y \in \xi \cup \{o\} \\ d(o,y)<d(o,x) }} d(x,y).
\]
In other words, $y$ is the radial nearest neighbour of $x$ precisely when $d(o,y)<d(o,x)$ and there are no points of {$\xi\cup\{o\}$} in the intersection of the balls $B(o, d(o, x)) \cap B(x, d(x,y))$. We remark that in the construction just described we implicitly assume that the radial nearest neighbours are unique, which in the random set-up below will indeed be the case.

In this paper we study the random radial spanning tree induced by a stationary {Poisson process} in a $d$-dimensional hyperbolic space {$\H^d(\kk)$ of constant sectional curvature $-\kk<0$, whose metric is denoted by $d_{\kk}$.} More precisely, let $\eta$ be a stationary {Poisson process} on $\H^d(\kk)$ with intensity $\gamma>0$, where stationary means that the law of $\eta$ is invariant under the full group of isometries of $\H^d(\kk)$. In particular, this implies that the intensity measure of $\eta$ is $\gamma$ times the $d$-dimensional Hausdorff measure $\cH^d_\kk$ on $\H^d(\kk)$. We fix once and for all an arbitrary point $o \in {\H^d(\kk)}$, referred to as the origin, and consider the radial spanning tree $\RST(\eta)$ on the full {Poisson process $\eta$,} as well as the radial spanning tree $\RST(\eta_R)$ on the restriction $\eta_R := \eta \cap {B^d(o, R)}$ of $\eta$ to the {open} hyperbolic ball ${B^d(o, R)}$ of radius $R > 0$, both with respect to the base point $o$. A simulation of such a radial spanning tree in the upper half-plane model can be seen in Figure \ref{fig:simulation}.
\begin{figure}
	\centering
	\includegraphics[scale=0.6]{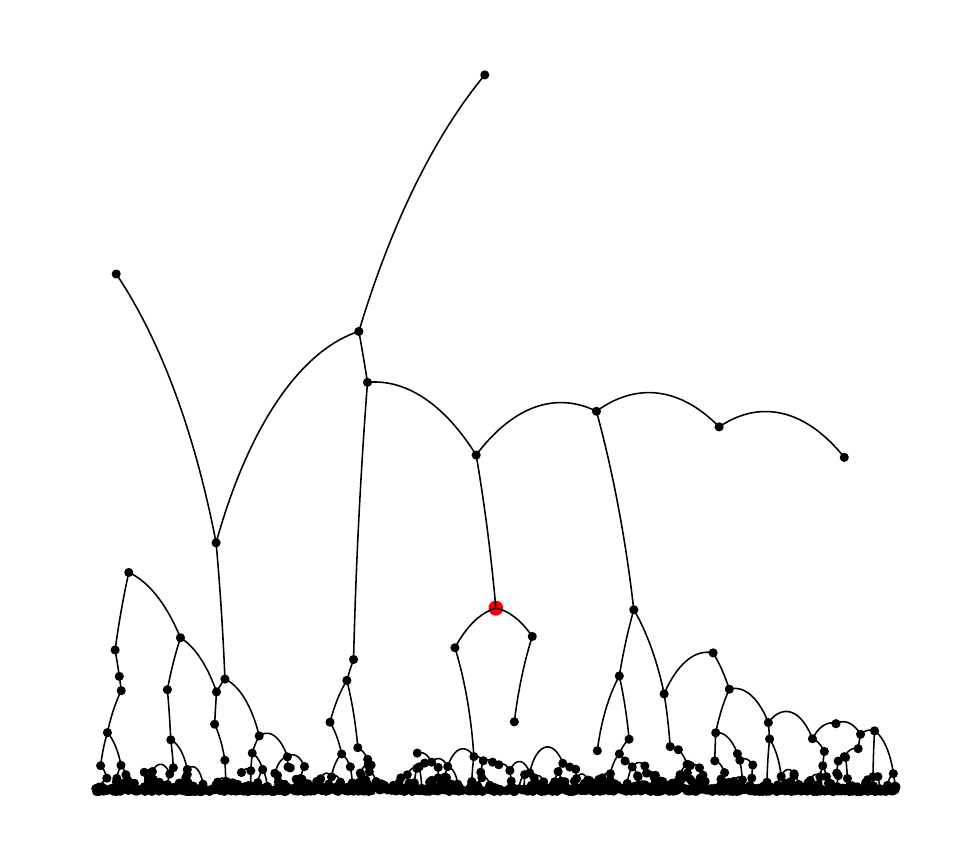}
	\caption{Simulation of an RST in $\mathbb{H}^2$ with the red point as the origin.}
	\label{fig:simulation}
\end{figure}

Our first result is concerned with the degree of the origin in the radial spanning tree $\RST(\eta)$, which we denote by $\deg(o)$. We remark that in Section \ref{subsec:ComparisonEuclidean} below we will compare our results with the ones for the corresponding construction in Euclidean space. For $k\in\N$ we let $\omega_{k}:={2\pi^{k/2}\over \Gamma(k/2)}$ be the surface content of the $(k-1)$-dimensional Euclidean unit sphere and $\kappa_k:=\omega_{k}/k$ be the volume of the $k$-dimensional Euclidean unit ball.

\begin{theorem}[Degree of the origin]\label{thm:DegreeOrigin}
	Consider the radial spanning tree $\RST(\eta)$ based on a stationary {Poisson process} $\eta$ on $\H^d(\kk)$ with intensity {$\gamma>0$} and $\kk>0$.
	\begin{itemize}
	\item[(a)] The degree of the origin has finite moments of all orders, that is, $\E [\deg(o)^n]<\infty$ for any $n\in\N$.

	\item[(b)] The expected degree of the origin is given by
 \begin{align}
		\E [\deg(o)] &= \gamma\omega_d \kk^{-{d\over 2}}\int_0^\infty \sinh^{d-1}(r) \nonumber\\
		&\quad\times\exp \left(-2\gamma\kappa_{d-1}  \kk^{-{d \over 2}} \int_0^{r/2} \left[{\cosh^2(r) \over \cosh^2(r-t)}-1\right]^{{d-1\over 2}} \,\dint t\right)\dint r. \label{eqn:expected_degree}
	\end{align}

	\item[(c)] The degree of the origin is not almost surely bounded, that is, for any $n\in\N$ it holds that $\PP(\deg(o)\geq n)>0$.
	\end{itemize}	
\end{theorem}

We turn to the study of metric properties of the radial spanning tree. {As before, for} {$R>0$ we denote by $\eta_R$ the restriction of $\eta$ to $B^d(o,R)$, i.e.\ $\eta_R$ are the points of $\eta$ belonging to $B^d(o,R)$. We are interested in the length of the edges from points of $\eta_R$ to their radial nearest neighbours.} For $\alpha\geq 0$ we denote by $\cL_{R,\gamma,\kk}^{(\alpha)}$ the {edge-length functional}
\[
\cL_{R,\gamma,\kk}^{(\alpha)} := \sum_{x \in \eta_R} \ell_\kk(x, {\eta})^\alpha = \sum_{x \in \eta_R} d_\kk(x, n(x, {\eta}))^\alpha,
\]
where we write {$\ell_\kk(x, \eta)$ for $d_\kk(x, n(x, \eta))$,} the hyperbolic distance between $x$ and its radial nearest neighbour. We are interested in the probabilistic behaviour of $\cL_{R,\gamma,\kk}^{(\alpha)}$, as $R\to\infty$. To simplify our presentation, we shall write 
\[
V_\kk(R)  := \cH_\kk^d\big({B^d(o,R)}\big)
\]
for the hyperbolic volume of a ball of radius $R$. We remark that it is well known that, up to constants, $V_\kk(R)$ grows like $e^{(d-1)\sqrt{\kk} R}$, as $R \to \infty$, see \eqref{eq:volume_ball} below. The following theorems are restricted to  $\alpha>0$ since for $\alpha=0$ one only counts the number of points {of $\eta$} in $B^d(o,R)$, {which is Poisson distributed with parameter $\gamma V_\kk(R)$.}

\begin{theorem}[Expectation and variance for {edge-length} functionals]\label{thm:ExpecationVariance}
Consider for $\alpha> 0$ and $R>0$ the edge-length functional $\cL_{R,\gamma,\kk}^{(\alpha)}$ of the radial spanning tree $\RST(\eta)$ based on a stationary Poisson process $\eta$ on $\H^d(\kk)$ with intensity $\gamma>0$ and $\kk>0$.
	\begin{itemize}
	\item[(a)] It holds that
	\begin{align}\label{eqn:expectation_length_functional}
	\lim\limits_{R\to\infty}\frac{\E[\cL_{R,\gamma,\kk}^{(\alpha)}]}{V_\kk(R)}&=\gamma\int_0^\infty e^{-\gamma \kk^{-d/2} G (\skk u^{1/\alpha})}\mathrm{d}u,
	\end{align}
	where $G(u)$ is given by
	\begin{equation*}
	G(u) := \kappa_{d-1}\int_0^u \left[2 e^{- t}\big(\cosh ( u) - \cosh ( t) \big)\right]^{{d-1\over 2}}\,\dint  t.
	\end{equation*}

	\item[(b)] The limit 
	\begin{align*}
		\VV^{(\alpha)}_{\gamma,\kk} := \lim\limits_{R\to\infty}\frac{\Var[\cL_{R,\gamma,\kk}^{(\alpha)}]}{V_\kk(R) }
	\end{align*}
	exists. Moreover, for all $c_0>0$, there are constants $\underline{c},\overline{c}>0$ depending on $\alpha$, $d$ and $c_0$ such that
	$$\underline{c} \leq \gamma^{{2\alpha\over d}-1} \,\VV^{(\alpha)}_{\gamma,\kk} \leq \overline{c}$$ for all $\gamma>0$ and $\kk>0$ with $\kk^{-d/2}\gamma\geq c_0$.
	\end{itemize}
\end{theorem}

\begin{rema}\rm 
An exact formula for the asymptotic variance constant $\VV^{(\alpha)}_{\gamma,1}$ will be given in Proposition \ref{prop:var_limit}. Combining this result with \eqref{eqn:asymptotic_variance_kk} provides an explicit formula for $\VV^{(\alpha)}_{\gamma,\kk}$ as well.
\end{rema}

Finally, we turn to the central limit theorem for the {edge-length functionals.} To provide a quantitative version of such a result, we consider both, the Wasserstein and the Kolmogorov distance. We recall that for two random variables $X$ and $Y$ the former is defined as
\[
d_W(X,Y) := \sup_{h} |\E [h(X)] - \E[ h(Y)]|,
\]
where the supremum is taken over all {functions $h : \R \to \R$ with Lipschitz constant at most one, if $\mathbb{E}[|X|],\mathbb{E}[|Y|]<\infty$,} while the latter is defined by
\[
d_K(X,Y) := \sup_{t \in \R} \left| \PP(X \leq t) - \PP(Y \leq t) \right|.
\]
With this notation, our central limit theorem takes the following form.

\begin{theorem}[Central limit theorem for edge-length functionals]\label{thm:CLT}
For $c_0>0$ and $\alpha>0$ consider the radial spanning tree $\RST(\eta)$ based on a stationary Poisson process $\eta$ on $\H^d(\kk)$ with $\kk>0$ and intensity $\gamma>0$ such that $\varkappa^{-d/ 2}\gamma \geq c_0$ and the edge-length functional $\cL_{R,\gamma,\kk}^{(\alpha)}$ for $R>0$. Denote by $N$ a standard Gaussian random variable. Then there exist constants $C>0$ and $r_0>0$ only depending on $d$, $\alpha$ and $c_0$ such that
	\begin{align*}
	d_\diamondsuit\left(\frac{\cL_{R,\gamma,\kk}^{(\alpha)}-\E[\cL_{R,\gamma,\kk}^{(\alpha)}]}{\sqrt{\Var[\cL_{R,\gamma,\kk}^{(\alpha)}]}},N\right)\leq \frac{C}{\sqrt{\gamma V_\kk(R) } }
	\end{align*}
	for $R\geq r_0$ and $\diamondsuit\in\{W,K\}$.
\end{theorem}

\begin{rema}
{\rm One can also study the edge-length functionals for $\alpha<0$. For the Euclidean case quantitative central limit theorems for $\alpha\in(-d/2,0)$ were derived in \cite{Trauthwein}. Note that for $\alpha<0$ short edges become crucial and that the restriction $\alpha>-\frac{d}{2}$ comes from the fact that second moments do otherwise not exist. We believe that similar results can be established for the hyperbolic case. As this requires different proofs than for $\alpha>0$, we refrained from considering this situation.}
\end{rema}

\subsection{Comparison with the Euclidean case}\label{subsec:ComparisonEuclidean}

Let us compare the results of the previous section with those available in the Euclidean case, which formally corresponds to the choice $\kk =0$ {and is, thus, excluded in our results.} For this, we consider the radial spanning tree in $\R^d$ with respect to the origin generated by a stationary {Poisson process $\eta^{\rm E}$ with intensity $\gamma>0$ in $\mathbb{R}^d$.} The mean degree of the origin in this Euclidean radial spanning tree is given by
\begin{equation}\label{eq:exp_deg_eucl}
\begin{split}
\E[ \deg^{\rm E}(o) ]&= \gamma\omega_d\int_0^\infty r^{d-1}\exp\left(-2\gamma\kappa_{d-1} \beta_d r^{d} \right)\,\dint r\\
									 &= {\omega_d\over 2d\kappa_{d-1}\beta_d} = {\sqrt{\pi}  \over 2\beta_d} \, { \Gamma\left({d \over 2} + {1 \over 2}\right)  \over  \Gamma\left({d \over 2} + {1}\right) },
\end{split}
\end{equation}
where {$2\kappa_{d-1}\beta_d$ is the volume of the intersection of two balls in $\mathbb{R}^d$ with radius one whose centres have distance one. This} can be seen by adapting the proof of Theorem \ref{thm:DegreeOrigin} (b) to the Euclidean set-up {and was shown for $d=2$ in \cite[Equation (7)]{BB}. The constant $\beta_d$ can be written as $\beta_d= {1\over 2} B_{3 \over 4}({d+1\over 2},{1\over 2})$ with the incomplete beta function $B_x(a,b) := \int_0^x t^{a-1}(1-t)^{b-1}\,\dint t$.} For example, $\beta_2={\pi\over 6}-{\sqrt{3}\over 8}$ and $\beta_3={5\over 24}$.
The Euclidean counterpart to the edge-length functional $\cL_{R,\gamma,\kk}^{(\alpha)}$ is
$$
\mathcal{L}^{(\alpha),\rm E}_{R,\gamma} := \sum_{x\in\eta_R^E}\ell_{\rm E}(x,\eta^{\rm E})^\alpha,
$$
where $B_R^{d,{\rm E}}$ is the centred Euclidean ball of radius $R>0$, $\eta_R^{\rm E}:=\eta^{\rm E}\cap B_R^{d,{\rm E}}$ is the restriction of $\eta^{\rm E}$ to $B_R^{d,{\rm E}}$ , $\ell_{\rm E}(x,\eta^{\rm E})$ denotes the Euclidean distance from $x$ to its radial nearest neighbour in $\eta^{\rm E}$ and $\alpha> 0$. The asymptotic behaviour for $R\to\infty$ of the expectation is given by
\begin{align}\label{eq:exp_length_eucl}
{\lim_{R \to \infty} \frac{\E[\mathcal{L}^{(\alpha),\rm E}_{R,\gamma}]}{V_{\rm E}(R)} = \gamma \int_0^\infty e^{-\gamma{\kappa_d\over 2}u^{d/\alpha}}\,\dint u } = \gamma\Big({2\over\gamma\kappa_d}\Big)^{\alpha/d}\Gamma\Big({\alpha\over d}+1\Big),
\end{align}
where $V_{\rm E}(R):= {\kappa_d} R^d$  is the volume of a Euclidean ball of radius $R$. This follows from rewriting \cite[(1.2) in Theorem 1.1]{ST}, where the $+1$ is missing in the argument of the Gamma function due to an inaccuracy in the last step of the proof. Note that the polynomial volume growth in \eqref{eq:exp_length_eucl} is in sharp contrast to the exponential growth of $V_\kk(R)$ for $\kk>0$ in the hyperbolic case.

We emphasise that all findings of the previous section do not cover the case $\kk=0$ and it is also a priori not clear that the Euclidean results can be derived by taking $\kk\to 0$. However, {letting $\kk\to 0$  in Theorem \ref{thm:DegreeOrigin} and Theorem \ref{thm:ExpecationVariance} (a) leads to \eqref{eq:exp_deg_eucl} and \eqref{eq:exp_length_eucl}, respectively. Indeed, since substitution yields
\begin{align*}
& \gamma\omega_d \kk^{-{d\over 2}}\int_0^\infty \sinh^{d-1}(r) \exp \left(-2\gamma\kappa_{d-1}  \kk^{-{d \over 2}} \int_0^{r/2} \left[{\cosh^2(r) \over \cosh^2(r-t)}-1\right]^{{d-1\over 2}} \,\dint t\right)\dint r \\
& = \gamma\omega_d \int_0^\infty \frac{\sinh^{d-1}(\sqrt{\kk} r)}{\kk^{d-1\over 2}}  \exp \left(-2\gamma\kappa_{d-1} \int_0^{r/2} \left[\frac{1}{\kk}\left({\cosh^2(\sqrt{\kk}r) \over \cosh^2(\sqrt{\kk}(r-t))}-1\right)\right]^{{d-1\over 2}} \,\dint t\right)\dint r
\end{align*}
and one has
$$
\lim_{\kk\to 0} \frac{\sinh(\sqrt{\kk}r)}{\sqrt{\kk}} = r \quad \text{and} \quad \lim_{\kk\to 0} \frac{1}{\kk}\left({\cosh^2(\sqrt{\kk}r) \over \cosh^2(\sqrt{\kk}(r-t))}-1\right) = r^2-(r-t)^2,
$$
one can use the dominated convergence theorem to show that the right-hand side of \eqref{eqn:expected_degree} converges to the expression in \eqref{eq:exp_deg_eucl}. The convergence of the right-hand side of \eqref{eqn:expectation_length_functional} to \eqref{eq:exp_length_eucl} follows from
$$
\lim_{u\to 0} \frac{G(u)}{u^d} = \frac{\kappa_d}{2}
$$
and a further application of the dominated convergence theorem.} For the sake of comparison, Table \ref{tab:DegreeOrigin} and Table \ref{tab:Length} show values of the expected degree of the origin and the asymptotic expected volume-normalised edge-length functional with $\alpha=1$ in the hyperbolic case with $\kk=1$, and the Euclidean case, for low dimensions $d$.

\begin{table}
	\begin{minipage}{0.45\columnwidth}
			\begin{equation*}
	\begin{array}{c | c | c | c}
	d &  \kk=1  & \kk=0  \text{ exact} & \kk=0  \text{ approx.}\\
	\hline 
	2 & 2.931  & {6\pi\over 4\pi-3\sqrt{3}} & 2.558\\
	3 & 3.985 & {16\over 5} & 3.200\\
	4 & 5.397 & {12\pi\over 8\pi-9\sqrt{3}} & 3.950\\ 
	5 & 7.332 & {256\over 53} & 4.832\\
	6 & 10.010 & {30\pi\over 20\pi-27\sqrt{3}} & 5.866\\
	7 & 13.749 & {2048\over 289} & 7.087\\
\end{array}
\end{equation*}
\caption{Expected degree of the origin in the radial spanning tree for $\kk=1$ (hyperbolic case) and $\kk=0$ (Euclidean case) and $d\in\{2,\ldots,7\}$.}
\label{tab:DegreeOrigin}
	\end{minipage}
	\hfill
	\begin{minipage}{0.45\columnwidth}
		\vspace*{0.47cm}
\begin{equation*}
	\begin{array}{c | c | c}
		d & \kk=1 & \kk=0 \\
		\hline 
		2 & 0.7591& 0.7071\\
		3& 0.7514 & 0.6979\\
		4  & 0.7762 & 0.7232\\
		5& 0.8087 & 0.7566\\
		6  & 0.8426 & 0.7920\\
		7 &  0.8761 & 0.8273\\
	\end{array}
\end{equation*}
	\caption{{Asymptotic} expected volume-normalised total {edge-length} (i.e.\ $\alpha=1$)  of the radial spanning tree for $\kk=1$ (hyperbolic case) and $\kk=0$ (Euclidean case) and $d\in\{2,\ldots,7\}$.}
	\label{tab:Length}
\end{minipage}
\end{table}

In Section \ref{sec:MainResults} we have approached the asymptotic geometry of edge-length functionals of the hyperbolic radial spanning tree by fixing the intensity $\gamma$ and the curvature $\kk$ and expanding the size of the ball in which the functional is observed. Alternatively, one can fix the size and the curvature and consider the asymptotics as $\gamma\to\infty$. From Proposition \ref{prop:Scaling} below it follows that $\cL_{1,\gamma,\kk}^{(\alpha)}$ has the same distribution as $\gamma^{-\alpha/d}\cL_{\gamma^{1/d},1,\gamma^{-2/d}\kk}$ so that letting the intensity grow is equivalent to increasing the observation window and letting the curvature parameter go to zero. Since $\kk=0$ corresponds to the Euclidean case, for $\gamma\to \infty$ and fixed $R>0$ the effects of the hyperbolic space might vanish and one can expect the same results as for the Euclidean space in \cite{ST}. This is further supported by the observation that for increasing $\gamma$ the edges become shorter and $\mathbb{H}^d$ can  locally be approximated by $\mathbb{R}^d$.

\section{Background material about hyperbolic geometry}\label{sec:background_hyperbolic}

\subsection{Hyperbolic space of curvature $-\kk$}
We denote by $\H^d(\kk)$ the $d$-dimensional hyperbolic space of curvature $-\kk$, i.e. the unique complete and simply connected Riemannian manifold of constant sectional curvature $-\kk$. Throughout this paper we fix the parameter $\kk > 0$ and often abbreviate $\H^d(1)$ by $\H^d$.
We will extensively use the upper half-space model for hyperbolic space (see, e.g., \cite[Chapter 3]{Lee}), which is the set
\[
\{x = (x_1, \ldots, x_d) \in \R^d \,:\, x_d > 0\}
\]
equipped with the hyperbolic Riemannian metric
\[
g_{\kk} := {1 \over \kk}\,{\dint x_1^2 + \cdots + \dint x_d^2 \over x_d^2}. 
\]
The hyperbolic geodesics in this model are realised by Euclidean rays and Euclidean circular arcs normal to the boundary hyperplane $\{x_d = 0\}$, see Figure \ref{fig:model_H2} for the case $d=2$. 
The Riemannian distance function is given by
\[
d_\kk(x,y) = {1 \over \sqrt{\kk}} \,{\rm arcosh}\bigg({1+}{|x-y|_{\rm E}^2  \over 2x_d y_d}\bigg),
\]
where $|\,\cdot\, |_{\rm E}$ {stands for} the Euclidean norm. 

The \emph{boundary} of hyperbolic space, denoted by  $\partial \H^d(\kk)$, in this  model is the  boundary hyperplane $\{x_d =0\}$ compactified by adding the point at $\infty$. In particular, topologically it is a $(d-1)$-dimensional sphere. We call points of $\partial\H^d(\kk)$ \emph{ideal points}. Observe that every geodesic ray ${\tau}\colon[t_0,\infty)\to\H^d(\kk)$ tends to a unique ideal point ${\tau}({\infty}) \in\partial\H^d(\kk)$.

	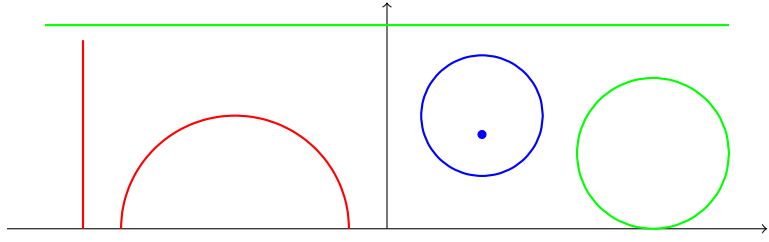
\begin{figure}[t]
	\centering
	\begin{tikzpicture}
		\draw[->] (-5,0) -- (5,0);
		\draw[->] (0,0) -- (0,3);
		\draw[red,thick] (-4,0) -- (-4,2.5);
		\draw[red,thick,domain=0:180,smooth] plot ({1.5*cos(\x)-2}, {1.5*sin(\x)});
		\draw[blue,thick,domain=0:360,smooth] plot ({0.8*cos(\x)+1.25}, {0.8*sin(\x)+1.5});
		\filldraw[blue] (1.25,1.25) circle (1.5pt);
		\draw[green,thick] (-4.5,2.7) -- (4.5,2.7);
		\draw[green,thick,domain=0:360,smooth] plot ({cos(\x)+3.5}, {sin(\x)+1});
	\end{tikzpicture}
	\caption{Hyperbolic lines (red), circles (blue) and horocycles (green) in the upper half-plane model for $\H^2$.}\label{fig:model_H2}
	\end{figure}

	Of particular importance to us are \emph{horospheres} in hyperbolic space, which are, informally speaking, spheres of infinite radius. These can be constructed as follows. Fix a geodesic ray ${\tau}(t)$ in $\H^d(\kk)$, and a point $p$ on ${\tau}$. For fixed $t $, consider the hyperbolic sphere $S_t$ centred at ${\tau}(t)$ and passing through $p$. As $t \to \infty$ the spheres $S_t$ converge to an unbounded hypersurface, called the horosphere passing through $p$ around the ideal point ${\tau}({\infty})$, see Figure \ref{fig:limit_circles}. For concreteness, in the upper half-space model of hyperbolic space, horospheres are realised as Euclidean spheres tangent to the boundary, or as Euclidean hyperplanes parallel to the boundary, see Figure \ref{fig:model_H2} for the case $d=2$ (where horospheres are traditionally called \emph{horocycles}).

\begin{figure}[t]
	\centering
	\begin{tikzpicture}
	\draw[->] (-5,0) -- (5,0);
	\draw[->] (0,0) -- (0,3);
	\draw[green,thick,domain=0:360,smooth] plot ({0.8*cos(\x)+1.25}, {0.8*sin(\x)+1.5});
	\draw[green,thick,domain=0:360,smooth] plot ({cos(\x)+1.25}, {sin(\x)+1.7});
	\draw[green,thick,domain=180:360,smooth] plot ({1.5*cos(\x)+1.25}, {1.5*sin(\x)+2.2});
	\draw[green,thick,domain=160:180,smooth,dashed] plot ({1.5*cos(\x)+1.25}, {1.5*sin(\x)+2.2});
	\draw[green,thick,domain=0:20,smooth,dashed] plot ({1.5*cos(\x)+1.25}, {1.5*sin(\x)+2.2});
	\draw[green,thick,domain=180+40:360-40,smooth] plot ({3*cos(\x)+1.25}, {3*sin(\x)+3.7});
	\draw[green,thick,domain=180+20:180+40,smooth,dashed] plot ({3*cos(\x)+1.25}, {3*sin(\x)+3.7});
	\draw[green,thick,domain=360-40:360-20,smooth,dashed] plot ({3*cos(\x)+1.25}, {3*sin(\x)+3.7});
	\draw[green,very thick,dashed] (-4.5,0.72) -- (4.5,0.72);
	\filldraw[black] (1.25,0.7) circle (1.5pt);
	\end{tikzpicture}
	\caption{Construction of a horocycle.}\label{fig:limit_circles}
\end{figure}
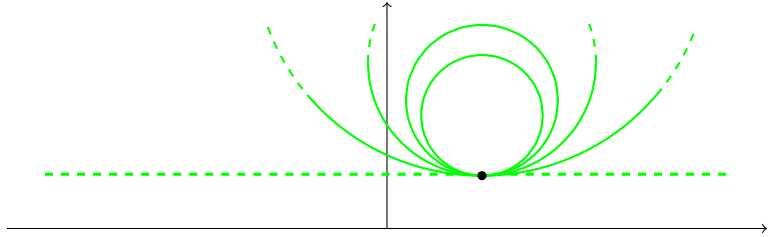

	Slightly more formally, horospheres may be alternatively defined in terms of \emph{Busemann functions}, see e.g. \cite{Eberlein}. Given a geodesic ray parametrised by arclength ${\tau}(s)$ in $\H^d(\kk)$, its associated Busemann function $B_{{\tau}} : \H^d(\kk) \to \R$ is defined by
	\begin{equation}
	B_{{\tau}}(p) := \lim_{s \to \infty} \left(  s-d_\kk(p, {\tau}(s))\right).
	\end{equation}
	The level sets of $B_{{\tau}}$ are exactly the horospheres around the ideal point ${\tau}({\infty})$ (observe that the level sets of the functions $p \mapsto   s-d_\kk(p, {\tau}(s))$ are hyperbolic spheres). For example, in the upper half-space model for $\H^d(1)$, the Busemann function associated with the geodesic ray ${\tau}(s) = (0, e^s)$, $s \ge 0$, is given by $B_{{\tau}}(x_1,\ldots,x_d) = \log x_d$, and hence horocycles  around the point at infinity ${\tau}({\infty})$ are Euclidean hyperplanes parallel to the boundary.
	
	Any horosphere $H$ is the boundary of two unbounded domains, only one of which is hyperbolically convex. We call this the \emph{horoball} bounded by $H$. It can be seen as the limit of the balls bounded by the spheres converging to $H$, or, which is the same, the appropriate sup-level set of the corresponding Busemann function. In the upper half-plane model, a horoball is realised as the Euclidean ball bounded by a Euclidean sphere tangent to the boundary, or as the  Euclidean half-space lying above a Euclidean hyperplane parallel to the boundary.

 \subsection{Integration formulas in hyperbolic space}

 In the sequel we will collect several integration formulas in $\H^d(\kk)$, which may be viewed as disintegration formulas for hyperbolic volume along spheres, hyperplanes, or horospheres. The first is  the  standard formula for polar integration in $\H^d(\kk)$.  Recall that $o \in \H^d(\kk)$ is a fixed origin and associate to each point $x \in \H^d(\kk) \setminus \{o\}$ the pair $(s,u) \in (0,\infty) \times S^{d-1}$ such that $s = d_\kk(o,x)$ and $u$ is the unit vector tangent to the (unique) geodesic starting at $o$ and passing through $x$. With this notation, we have for every measurable $f \colon \H^d(\kk) \to [0,\infty)$ the formula
\begin{equation}\label{eq:polar_integration}
\int_{\H^d(\kk)} f(x) \,\cH_\kk^d(\dint x) = \kk^{-{d-1 \over 2}}\int_0^\infty \int_{S^{d-1}} f(s,u) \sinh^{d-1}(\sqrt{\kk}s) \,\dint u\,\dint s, 
\end{equation} 
where we denote by $\dint u$ the volume element of the standard rotationally-invariant volume form on the unit sphere $S^{d-1}$ in the tangent space of $\H^d(\kk)$ at $o$. In particular, we get the following formula for the volume of a hyperbolic ball:
\begin{equation}\label{eq:volume_ball}
\begin{split}
 V_\kk(R) = \cH_\kk^d(B^d(x,R)) &= \omega_d\kk^{-{d-1\over 2}}\int_0^R\sinh^{d-1}(\sqrt{\kk}\,u)\,\dint u \\
& =\omega_{d}  \kk^{-{d\over 2}}\int_0^{\sqrt{\kk}R} \sinh^{d-1}(s) \,\dint s.
\end{split}
\end{equation}
We note the relation
\begin{equation}\label{eq:Vk-rescale}
V_\kk(R) = \kk^{-d/2} V_1(\sqrt{\kk} R).
\end{equation}

The next formula is a disintegration of the hyperbolic volume along hyperplanes.
	
	\begin{lemma}\label{lem:disintegration_planes}
	Let $L \subset \H^d(\kk)$ be a complete geodesic with an arclength parametrisation ${\tau}(t)$. Denote by $H_t$, $t \in \R$, the hyperbolic hyperplane meeting $L$ at ${\tau}(t)$ orthogonally. Then for any measurable function $f : \H^d(\kk) \to [0,\infty)$ one has that
	\begin{equation*}
		\int_{\H^d(\kk)} f(x) \,\cH_\kk^d(\dint x) = \int_{\R} \int_{H_t} f(x) \cosh \big(\sqrt{\kk}\,d_\kk(x,L)\big) \,\cH_\kk^{d-1}(\dint x) \,\dint t,
	\end{equation*}
	in which $d_\kk(x,L):=\min_{y\in L}d_\kk(x,y)$ is the hyperbolic distance between $x$ and $L$.
	\end{lemma}
	
\begin{proof}
Let $\Pi:\H^d(\kk) \to L$ be the orthogonal projection onto $L$. By definition, the projection $\Pi(x) \in L$ of a point $x \in \H^d(\kk)$ is the endpoint in $L$ of the (unique) geodesic segment starting from $x$ and hitting $L$ orthogonally. Define $F : \H^d(\kk) \to \R$ by $F := {\tau}^{-1} \circ \Pi$. That is, for $x \in \H^d(\kk)$ one has $F(x)=t$ precisely when $\Pi(x) = {\tau}(t) $. Moreover, observe that the fibres of $\Pi$ are hyperplanes orthogonal to $L$, and so $F^{-1}(t) =H_t$. Therefore the coarea formula \cite[Theorem VIII.3.3]{Chavel} gives 
\begin{equation*}
\int_{\H^d(\kk)} f(x) \,\cH_\kk^d(\dint x) = \int_\R \int_{H_t} {f(x) \over \|\nabla F(x) \|_{g_\kk}} \,\cH_\kk^{d-1}(\dint x)\, \dint t,
\end{equation*}
where $\|\cdot \|_{g_\kk}$ denotes the norm associated with the Riemannian metric $g_\kk$.
Hence to prove the claim it remains to show that $\|\nabla F (x) \|_{g_\kk} = 1/\cosh\big(\sqrt{\kk}\,d_\kk(x,L)\big)$. This is readily seen in the upper half-space model as follows. We may assume that $L$ is the geodesic along the $x_d-$axis. Its arclength parametrisation is
\begin{equation*}
{\tau}(t) = (0,\ldots, 0,e^{\sqrt{\kk}t}).
\end{equation*}
Moreover, the geodesic from a point $x \in \H^d(\kk)$ orthogonal to $L$ is then a Euclidean circular arc with centre at $(0,\ldots, 0)$ and radius $|x|_{\rm E}$, where we recall that $|\,\cdot\,|_{\rm E}$ denotes the Euclidean norm. In particular,
\begin{equation*}
\Pi(x) = (0,\ldots, 0, |x|_{\rm E}),
\end{equation*}
see Figure \ref{fig:Pi}.
	\begin{figure}
	\centering
	\begin{tikzpicture}[scale=1.2]
	\draw[thick] (-3,0) -- (3,0);
	\draw[thick, -] (0,0) -- (0,3) node[above]{$\footnotesize L$};
	\draw[thick,domain=0:90,smooth,dashed] plot ({2 * cos(\x)}, {2 * sin(\x)});
	\draw[domain=45:90,smooth] plot ({0.5 * cos(\x)}, {0.5 * sin(\x)});
	\node at (0.7*0.3826, 0.7*0.9238) {$\footnotesize\phi$};
	\draw[] (0,0)--(1.414,1.414) node[above right]{$\footnotesize x$};
	\filldraw[thick,color=black] (1.414,1.414) circle (1pt);
	\filldraw[thick,color=black] (0,2) circle (1pt) node[left]{$\footnotesize\Pi(x)$};
	\end{tikzpicture}
	\caption{The orthogonal projection onto $L$.}\label{fig:Pi}
\end{figure}
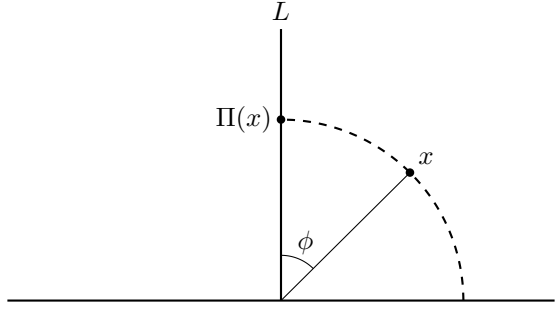
Therefore,
\begin{equation*}
F(x) = {1 \over \sqrt{\kk}} \,\log |x|_{\rm E} .
\end{equation*}
Denoting for $i,j\in\{1,\ldots,d\}$ by $(g_\kk)^{ij}(x) = \kk x_d^2 \,\delta^{ij}$ the entries of the inverse of the metric tensor, we find that 
\begin{align*}
\|\nabla F\|_{g_\kk}^2 = \sum_{i,j=1}^d (g_\kk)^{ij} {\partial F \over \partial x_i}{\partial F \over \partial x_j} = \kk x_d^2 \sum_{j=1}^d\left( {\partial F \over \partial x_j}\right)^2 = \left({x_d \over |x|_{\rm E}}\right)^2 = \cos^2\phi,
\end{align*}
where $\phi$ is the angle from the $x_d$-axis to the straight line passing through $(0,\dots,0)$ and $x$, see Figure \ref{fig:Pi}. On the other hand, one computes easily (see e.g. \cite[Figure 24]{CFKP}) that 
$$d_\kk(x, L) = d_\kk(x, \Pi(x)) = {1 \over \sqrt{\kk}}\,\arcosh\left({1 \over \cos \phi}\right).$$
Combining these computations, we deduce that
\begin{equation*}
\|\nabla F(x)\|_{g_\kk} = \cos \phi = {1\over \cosh\big(\sqrt{\kk}\,d_\kk(x, L)\big)},
\end{equation*}
as desired. This completes the proof.
\end{proof}

Finally, we will also employ in this paper the  following disintegration of the hyperbolic volume with respect to horospheres. We will call horospheres centred at the same ideal point \emph{parallel}. The family of all parallel horospheres corresponding to some ideal point covers the entire hyperbolic space, and it is convenient to parametrise the family in terms of the signed distance from a given horosphere $H$ to the origin $o \in \H^d(\kk)$, where we take the distance to be \emph{negative} if $o$ lies inside the horoball bounded by $H$,	and positive otherwise.

	\begin{lemma}\label{lem:disintegration_horo}
	Let $(H_t)_{t \in \R}$ be a family of parallel horospheres in $\H^d(\kk)$, parametrised such that $H_t$ has signed distance $t$ from the origin $o \in \H^d(\kk)$. Then for any measurable function $f : \H^d(\kk) \to [0,\infty)$ one has
	\begin{equation*}
	\int_{\H^d(\kk)} f(x) \,\cH_\kk^d(\dint x) = \int_{\R} \int_{H_t} f(x) \,\cH_\kk^{d-1}(\dint x) \,\dint t.
	\end{equation*}
		\end{lemma}
	
	\begin{proof}
	Let ${\tau}$ be the complete geodesic corresponding to the parallel family, such that ${\tau}(0) \in H_0$, and let $B_{{\tau}}$ be the associated Busemann function. By the coarea formula \cite[Theorem VIII.3.3]{Chavel} we have
	\begin{equation}\label{eq:coarea}
		\int_{\H^d(\kk)} f(x) \,\cH_\kk^d(\dint x) = \int_{\R} \int_{B_\tau^{-1}(t)} {f(x) \over \|\nabla B_{{\tau}} (x) \|_{g_\kk}} \,\cH_\kk^{d-1}(\dint x) \,\dint t,
	\end{equation}
	where we recall that $\|\cdot \|_{g_\kk}$ denotes the norm associated with the Riemannian metric $g_\kk$.

	It is known from \cite[Proposition 1.10.2.(2)]{Eberlein} that $\|\nabla B_{{\tau}}\|_{g_\kk} \equiv 1$. Thus it remains to verify that $B_{{\tau}}^{-1}(t)=H_t$. Indeed, since ${\tau}$ is orthogonal to all horospheres centred at ${\tau}({\infty})$, one has that
 	\begin{equation*}
 	d_\kk(B_{{\tau}}^{-1}(t), B_{{\tau}}^{-1}(0))=d_\kk({\tau}(t),{\tau}(0))=|t|.
 	\end{equation*}
 	 Finally, as the horoball bounded by $B_{{\tau}}^{-1}(t)$ is $B_{{\tau}}^{-1}([t,\infty))$, it contains $o$ if and only if $t \leq  0$, we conclude that $B_{{\tau}}^{-1}(t) = H_t$. Plugging all this back into \eqref{eq:coarea} completes the proof.
	\end{proof}

\section{Background material on {Poisson processes}}\label{sec:PPP}

Let $\XX$ be a Polish space with Borel $\sigma$-field $\cB(\XX)$, which is supplied with a {locally} finite measure $\mu$ having no atoms. By $\mathcal{F}_{\rm lf}(\XX)$ we denote the space of locally finite subsets of $\XX$. By a Poisson process $\eta$ on $\XX$ with intensity measure $\mu$ we understand a random element in $\mathcal{F}_{\rm lf}(\XX)$, defined on some auxiliary probability space $(\Omega,\mathcal{A},\PP)$ with the following two properties:
\begin{itemize}
\item[(i)] For any $B\in\cB(\XX)$ the random variable {$|\eta\cap B|$, which is the cardinality of the set $\eta\cap B$,} has a Poisson distribution with mean $\mu(B)$;
\item[(ii)] for any $k\geq 2$ and any $k$ disjoint subsets $B_1,\ldots,B_k\in\cB(\XX)$ the random variables {$|\eta\cap B_1|,\ldots,|\eta\cap B_k|$} are independent.
\end{itemize}
It is convenient to identify the locally finite random set $\eta$ with the random counting measure $\sum_{x\in\eta}\delta_x$, where $\delta_x$ is the Dirac measure concentrated at $x$. In other words, a Poisson process can be regarded as a random element in the space $\mathcal{N}_{\rm lfc}(\XX)$ of locally finite counting measures on $\XX$. We keep the notation $\eta$ for this random measure and write $\eta(B)$ instead of {$|\eta\cap B|$} in what follows. As usual in point process theory, we switch freely between both descriptions of $\eta$. Moreover, we refer to \cite{LP} for further background material on general Poisson processes.

One of the main tools in the analysis of Poisson processes is the so-called Mecke equation \cite[Theorem 4.1]{LP}. It says that for any non-negative measurable function $f:\XX\times\mathcal{N}_{\rm lfc}(\XX)\to\R$	one has that
\begin{equation}\label{eq:Mecke}
\E\bigg[\sum_{x\in\eta}f(x,\eta)\bigg] = \int_\XX\E {[}f(x,\eta+\delta_x){]}\,\mu(\dint x).
\end{equation}
An iteration of the last identity shows that if $k\geq 2$ and if $f:\XX^k\times\mathcal{N}_{\rm lfc}(\XX)\to\R$ is a non-negative measurable function, then
\begin{align}\label{eq:Meckemulti}
\notag &\E\bigg[\sum_{(x_1,\ldots,x_k)\in\eta_{\neq}^k} f(x_1,\ldots,x_k,\eta)\bigg] \\
&= \int_{\XX}\cdots\int_{\XX}\E{[} f(x_1,\ldots,x_k,\eta+\delta_{x_1}+\ldots+\delta_{x_k}){]}\,\mu(\dint x_1)\ldots\mu(\dint x_k),
\end{align}
where $\eta_{\neq}^k$ is the set of $k$-tuples of distinct points of $\eta$, see \cite[Theorem 4.4]{LP}.

Another tool we shall use is the so-called Poincar\'e inequality for Poisson functionals \cite[Theorem 18.7]{LP}. By the latter we mean random variables of the form $F=F(\eta)$, that is, random variables that can be described only in terms of the {Poisson process} $\eta$. Assuming that $\E {[}F^2{]}<\infty$ one has that
\begin{equation}\label{eq:Poincare}
\Var [F] \leq \int_\XX \E[(D_xF)^2]\,\mu(\dint x),
\end{equation}
where $D_xF:=F(\eta+\delta_x)-F(\eta)$ is the first-order difference operator with respect to $x$ and applied to $F$. 
With the additional condition that we can find some constant ${a} \geq 0$ for which
\begin{equation*}
	\int_{\XX}\int_{\XX} \E\left[ (D^2_{x_1, x_2} F)^2\right] \, \mu(\dint x_1) \,\mu(\dint x_2) \leq  {a}  \int_{\XX} \E [\left(D_x F\right)^2]\,\mu(\dint x) < {\infty} 
\end{equation*}
holds, a lower bound for the variance is by \cite[Theorem 1.1]{STr} given by
\begin{align}\label{eq:reversePoincare}
	\Var[F] \geq \frac{4}{({a} + 2)^2} \int_{\XX} \E [\left(D_x F\right)^2]\,\mu (\dint x).
\end{align}
Here,
\begin{align*}
	D^2_{x_1, x_2} F:=D_{x_1}(D_{x_2}F)=F(\eta+\delta_{x_1}+\delta_{x_2})-F(\eta+\delta_{x_1})-F(\eta+\delta_{x_2})+F(\eta)
\end{align*}
  is the iterated, second-order difference operator for $x_1,x_2\in\XX$.

To show the central limit theorem for edge-length functionals we also use the following theorem from \cite[Theorem 6.1]{LPS}, which yields quantitative bounds for normal approximation.

\begin{theorem}\label{thm:normalapproximation_bounds_dW_dK}
	Let $F$ be a Poisson functional satisfying  $\int_\XX \E[\left(D_x F\right)^2]\,\mu (\dint x)<\infty$, $\Var[F]>0$ {and} the moment conditions
	\begin{align*}
		\E [\lvert D_xF\rvert^{5}]&\leq \bar{c} \quad\text{for}\quad \mu\text{-a.e.}\;  x\in\mathbb{X},\\
		\E [\lvert D^2_{x_1,x_2}F\rvert^{5}]&\leq \bar{c}\quad\text{for}\quad \mu^2\text{-a.e.}\;  (x_1,x_2)\in\mathbb{X}^2
	\end{align*}
	for some constant $\bar{c}\geq 1$. Denote by $N$ a standard Gaussian random variable. Then 
	\begin{align*}
		{d_\diamondsuit}\Bigg(\frac{F-\E[F]}{\sqrt{\Var[F]}},N\Bigg)\leq&
		\frac{5\bar{c}}{\Var[F]}\Bigg(\int_\XX\Bigg(\int_\XX\p(D^2_{x_1,x_2}F\neq 0)^{1/20}\mu(\dint x_2)\Bigg)^2\mu(\dint x_1)\Bigg)^{1/2}\\
		&+\frac{\bar{c}I_F^{1/2}}{\Var[F]}+\frac{2\bar{c}I_F}{\Var[F]^{3/2}}+\frac{\bar{c}I_F^{5/4}+2\bar{c}I_F^{3/2}}{\Var[F]^2}\\
		&+\frac{\sqrt{6}\bar{c}+\sqrt{3}\bar{c}}{\Var[F]}\Big(\int_\XX\p(D_{x_1,x_2}^2F\neq 0)^{1/10}\mu^2(\dint (x_1,x_2))\Big)^{1/2}
	\end{align*}
for $\diamondsuit\in\{W,K\}$,	where 
	\begin{align*}
		I_F:=\int_\XX\p(D_xF\neq 0)^{1/10}\mu(\dint x).
	\end{align*}
\end{theorem}

	\section{Degree of the origin: proof of Theorem \ref{thm:DegreeOrigin}}\label{sec:proofs_origin}
	
	In this section we analyse the degree of the origin in the radial spanning tree $\RST(\eta)$.  We start with the exact formula for the expected degree.
	
	\begin{proof}[Proof of Theorem \ref{thm:DegreeOrigin} (b)]
		By definition of the radial nearest neighbour and the Mecke equation \eqref{eq:Mecke} for {Poisson processes} we have that
		\begin{align*}
		\E[\deg(o)] &= \E\Big[\sum_{x\in\eta}\mathbbm{1}\{\eta(B^d(x,d_\kk(o,x))\cap B^d(o,d_\kk(o,x)))=0\}\Big]\\
		&=\gamma\int_{\H^d(\kk)} \mathbb{P}((\eta+\delta_x)(B^d(x,d_\kk(o,x))\cap B^d(o,d_\kk(o,x)))=0)\,\cH_\kk^d(\dint x)\\
		&=\gamma\int_{\H^d(\kk)}\exp(-\gamma\cH_\kk^d(B^d(x,d_\kk(o,x))\cap B^d(o,d_\kk(o,x))))\,\cH_\kk^d(\dint x).
		\end{align*}
		Next, we {apply} hyperbolic spherical coordinates as in \eqref{eq:polar_integration} to see that 
		\begin{align}\label{eq:130523A}
		\E{[}\deg(o){]} &= \gamma\kk^{-{d-1 \over 2}}\omega_d \int_0^\infty \exp(-\gamma\cH_\kk^d(B^d(x_r,r)\cap B^d(o,r)))\,\sinh^{d-1}(\sqrt{\kk}r)\,\dint r,
		\end{align}
		where $x_r\in\H^d(\kk)$ is an arbitrary point satisfying $d_\kk(o,x_r)=r$.
		
		\begin{figure}[t]
			\centering
			\begin{tikzpicture}[scale=1.5]
			\begin{scope}
			\clip(-0.5,-1) rectangle (0.5,1);
			\filldraw[thick,color=blue!50] (1,0) circle (1.0);
			\end{scope}
			\draw[thick] (1,0) circle (1.0);
			\draw[thick] (0,0) circle (1.0);
			\draw[thick, dotted] (0.5,0.866)--(0.5,-0.866);
			\draw[thick,<->] (0,-1.2) -- node [below] {\footnotesize $r/2$}  (0.5,-1.2);
			\filldraw[black] (0,0) circle (1pt);
			\filldraw[black] (1,0) circle (1pt);
			\end{tikzpicture}
			\qquad 
			\begin{tikzpicture}[scale=1.5]
			\begin{scope}
			\clip(-1,0.5) rectangle (1,1);
			\filldraw[color=blue!50](0,0) circle (1.0);
			\end{scope}
			\filldraw[black] (0,0) circle (1pt);
			\draw[thick](0,0) circle (1.0);
			\draw[thick,dotted](-0.866,0.5)--(0.866,0.5);
			\draw[thick](0,-1.2)--(0,1.2)node[above]{\footnotesize \blue{$L$}};
			\draw[thick](-1.2,0.7)--(1.2,0.7) node[right]{\footnotesize$H_t$};
			\draw[](0,0)--node[below right]{\footnotesize$r$}(0.714,0.7);
			\draw[<->](-.1,0)--node[left]{\footnotesize$r-t$}(-.1,0.69) ;
			\end{tikzpicture}
			\caption{Left: intersection of two balls. Right: computing the volume of a cap.}\label{fig:caps}
		\end{figure}
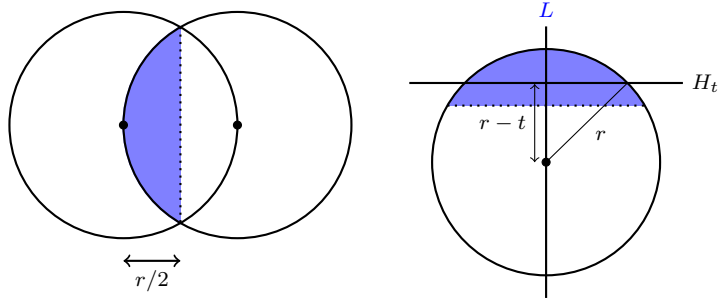

		To compute  the hyperbolic volume of the intersection $B^d(x_r,r)\cap B^d(o,r)$, first observe that the intersection is the union of two caps of height $r/2$ in a hyperbolic ball  {$B^d(o,r)$} of radius $r$, see the left panel of Figure \ref{fig:caps}. The volume of such a cap can be computed using Lemma \ref{lem:disintegration_planes}, see the right panel of Figure \ref{fig:caps}. We choose for $L$ as in Figure \ref{fig:caps} the radial line of the ball passing through the apex of the cap, and set the arclength parameter at the apex to zero so that
$$
\cH_\kk^d(B^d(x_r,r)\cap B^d(o,r))= 2 \int_0^{r/2} \int_{H_t \cap  {B^d(o,r)}} \cosh \big(\sqrt{\kk}\,d_\kk(x,{L})\big)\, \cH_\kk^{d-1}(\dint x) \,\dint t.
$$
Then the intersection $H_t \cap {B^d(o,r)}$ of the ball with the hyperplane $H_t$ orthogonal to $L$ at distance $t$ from the apex is a $(d-1)$-dimensional hyperbolic ball within $H_t$, whose radius is computed by the hyperbolic Pythagorean theorem \cite[Equation (12.97)]{Coxeter} to be
		$$
		{1 \over \sqrt{\kk}}\arcosh \left( {\cosh (\sqrt{\kk}r) \over \cosh (\sqrt{\kk}(r-t))} \right).
		$$
Together with the hyperbolic spherical coordinates from \eqref{eq:polar_integration} within $H_t$, this gives
		\begin{align*}
		& \cH_\kk^d(B^d(x_r,r)\cap B^d(o,r)) \\
		& = 2\omega_{d-1} \kk^{-{d-2\over 2}} \int_0^{r/2} \int_0^{{1 \over \sqrt{\kk}}\arcosh\left( {\cosh (\sqrt{\kk}r) \over \cosh (\sqrt{\kk}(r-t))}\right) } \cosh (\sqrt{\kk }s) \sinh^{d-2}(\sqrt{\kk}s)\,\dint s \,\dint t \\
		&=2 \omega_{d-1} \kk^{-{d-1 \over 2}} \int_{0}^{r/2} \frac{1}{d-1} \sinh^{d-1} \left(\arcosh \left({\cosh (\sqrt{\kk}r) \over \cosh (\sqrt{\kk}(r-t))}\right)\right) \,\dint t \\
		& = {2\kappa_{d-1}}  \kk^{-{d-1 \over 2}} \int_0^{r/2} \left[{\cosh^2(\sqrt{\kk}r) \over \cosh^2(\sqrt{\kk}(r-t))}-1\right]^{{d-1\over 2}} \,\dint t,
		\end{align*}
		where we have used that $\omega_{d-1} = (d-1) \kappa_{d-1}$ and $\sinh^2(u)=\cosh^2(u)-1$ for $u\in\mathbb{R}$. Plugging this back into \eqref{eq:130523A} and making the substitution $(r,t) \mapsto (\skk \,r, \skk\, t)$ completes the proof of part (b) of Theorem \ref{thm:DegreeOrigin}.
	\end{proof}
	
	Next, we show the boundedness of all moments of $\deg(o)$.
	
	\begin{proof}[Proof of Theorem \ref{thm:DegreeOrigin} (a)]
	For $n=1$ the result is clear by the previous proof of Theorem \ref{thm:DegreeOrigin} (b). For $n\geq 2$ consider the $n$-th factorial moment of $\deg(o)$. To emphasise the dependence of this random variable on the underlying intensity $\gamma$, we write $\deg_\gamma(o)$. Using the multivariate Mecke equation \eqref{eq:Meckemulti} for Poisson processes, we obtain
	\begin{align*}
	&\E[\deg_\gamma(o)(\deg_\gamma(o)-1)\cdots(\deg_\gamma(o)-n+1)]\\
	& {= \gamma^n\int_{\H^d(\kk)}\cdots\int_{\H^d(\kk)} \mathbb{P}(x_1,\ldots,x_n\text{ neighbours of $o$ in }\operatorname{RST}(\eta+\delta_{x_1}+\ldots+\delta_{x_n}))} \\
	& \hspace{3.5cm}{\,\cH_\kk^d(\dint x_1)\ldots\cH_\kk^d(\dint x_n) } \\
	&\leq \gamma^n\int_{\H^d(\kk)}\cdots\int_{\H^d(\kk)}\exp(- H_\gamma(x_1,\ldots,x_n))\,\cH_\kk^d(\dint x_1)\ldots\cH_\kk^d(\dint x_n)
	\end{align*}
	with
	$$
	H_\gamma(x_1,\ldots,x_n) := \gamma\cH^d_\kk\Big(\bigcup_{i=1}^n(B^d(x_i,d(o,x_i))\cap B^d(o,d(o,x_i)))\Big)
	$$
	since $x_1,\hdots,x_n$ can be only neighbours of $o$ if $\eta$ has no points in the union set considered in $H_\gamma$. Together with 
	\begin{align*}
	H_\gamma(x_1,\ldots,x_n) 
	&\geq{\gamma \max_{i\in\{1,\dots,n\}}\cH_\kk^d(B^d(x_i,d_\kk(o,x_i))\cap B^d(o,d_\kk(o,x_i)))}\\
	&\geq{\gamma\over n}\sum_{i=1}^n\cH_\kk^d(B^d(x_i,d_\kk(o,x_i))\cap B^d(o,d_\kk(o,x_i)))
	\end{align*}
	this shows that
	\begin{align*}
		&\E[\deg_\gamma(o)(\deg_\gamma(o)-1)\cdots(\deg_\gamma(o)-n+1)] \\
		&\leq n^n\Big({\gamma\over n}\int_{\H^d_\kk}\exp\Big(-{\gamma\over n}\cH_\kk^d(B^d(x,d_\kk(o,x))\cap B^d(o,d_\kk(o,x)))\Big)\cH_\kk^d(\dint x)\Big)^n\\
		&=n^n\big(\E[\deg_{\gamma/n}(o)]\big)^n.
	\end{align*}
	Since $\E[\deg_{\gamma/n}(o)]<\infty$ by part (b) of Theorem \ref{thm:DegreeOrigin}, $\E[\deg_\gamma(o)(\deg_\gamma(o)-1)\cdots(\deg_\gamma(o)-n+1)] <\infty$ and hence $\E[\deg_\gamma(o)^n]<\infty$ as well. This completes the proof of part (a).
	\end{proof}
	
	Finally, we consider the question of boundedness of the degree of the origin.
	
	\begin{proof}[Proof of Theorem \ref{thm:DegreeOrigin} (c)]
		
		 We prove that for every $n \geq 3$ one has $\PP(\deg(o) \geq n) > 0$, which obviously implies the result. Let us fix an arbitrary hyperbolic $2$-plane $E$ passing through $o$.
		 The proof relies on the existence of a hyperbolic regular $n$-gon within $E$ whose side length $s$ is larger than the radius $r$ of its circumscribing circle. Such an $n$-gon may be constructed by gluing together $n$ isosceles triangles with side length $r$, base length $s$ and apex angle $\frac{2\pi}{n}$, see the left panel of Figure \ref{fig:n-gon}. 

\begin{figure}
    \centering
    \begin{minipage}[c]{0.45\textwidth}
        \centering
        			\begin{tikzpicture}[scale=1]
  \def\numvertices{6}
  \def\radius{2.5cm} 
  \foreach \i in {1,...,\numvertices} {
    \coordinate (v\i) at ({90 + (\i - 1) * (360/\numvertices)}:\radius);
  }

  \foreach \i in {1,...,\numvertices} {
    \pgfmathtruncatemacro{\next}{mod(\i, \numvertices) + 1}
    \draw[dashed] (v\i) -- (v\next);
  }

  \draw[black, line width=0.6pt] (v5) -- (0, 0) -- (v6) -- cycle;

  \coordinate (m) at ($(v5)!0.5!(v6)$);
  \node[right] at (m) {$s$};

  \draw[dashed] (0, 0) -- (v5);
  \node[below] at ($(0, 0)!0.5!(v5)$) {$r$};

  \coordinate (O) at (0, 0);
  \draw pic[draw, black, angle radius=0.7cm, "$\frac{2\pi}{n}$", angle eccentricity=1.4] {angle=v5--O--v6};
			\end{tikzpicture}
    \end{minipage}\hfill
    \begin{minipage}[c]{0.45\textwidth}
        \centering
         			\begin{tikzpicture}[scale=1]
			  \def\numvertices{6}
  \def\radius{2.5cm} 
  
  \coordinate (C) at ({90 + (5 - 1) * (360/\numvertices)}:\radius);
  \coordinate (B) at ({90 + (6 - 1) * (360/\numvertices)}:\radius);
  \coordinate (A) at (0, 0);
  \coordinate (M) at ($(B)!0.5!(C)$);
  
  \draw[line width=.6pt] (A) -- node[above left] {$r$} (B) -- node[right] {$\frac{s}{2}$} (M) -- (C) -- cycle;

  \draw[line width=.6pt] (A) -- (M);


  \draw pic[draw, black, angle radius=0.7cm, "$\frac{\pi}{n}$", angle eccentricity=1.4] {angle=M--A--B};
			\end{tikzpicture}
    \end{minipage}
\caption{Left: Constructing a regular $n$-gon; Right: The hyperbolic law of sines: $\sin {\pi\over n}={\sinh (s/2)\over\sinh r} $.}\label{fig:n-gon}
\end{figure}
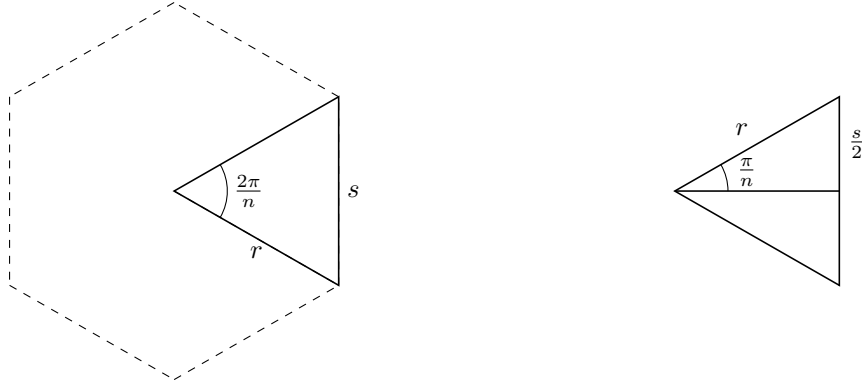

By the hyperbolic law of sines \cite[Equation (12.96)]{Coxeter} one computes, see the right panel of Figure \ref{fig:n-gon},
		 \begin{equation}\label{eq:s}
		 s = {2\over \skk} \operatorname{arsinh} \left(\sinh (\skk r) \cdot \sin {\pi\over n}\right) \geq {2\over \skk} \log \left(2\sinh (\skk r) \cdot \sin {\pi\over n}\right),
		 \end{equation}
which follows from $ \operatorname{arsinh} (x)=\log(x+\sqrt{x^2+1})\geq \log(2x)$ for any $x>0$.
	 		
For $r =  \big(1+\frac{2}{\skk}\big)\log(n)$ we have, using in the second step the inequality $\sin x \geq {2\over \pi}x$ for $x\in [0, {\pi\over 2}]$ (which follows from concavity of $x \mapsto \sin x$ on $[0,{\pi\over 2}]$) with $x={\pi\over n}$,
	 	\begin{align*}
	 		s&\geq {2\over \skk} \log \left(2\sinh ((\skk+2)\log(n)) \cdot \sin {\pi\over n}\right)\\
	 		& \geq {2\over \skk} \log \left( \Big(n^{\skk+2}-n^{-(\skk+2)}\Big) \cdot{2\over n} \right) \\
	 		& = {2\over \skk} \log \left( n^{\skk+1}\cdot 2 \big(1-n^{-2(\skk + 2)}\big) \right) \\
	 		&={2\over \skk} \left[ (1+\skk) \log n + \log\left( 2\big(1-n^{-2(\skk+2)} \big) \right)\right] \\
	 		& \geq  \left(2+\frac{2}{\skk}\right)\log(n) > r,
	 	\end{align*}
	 	as desired, where in the penultimate inequality we used that $2\big(1-n^{-2(\skk+2)} \big)\geq 1$.
		 
Now consider within the $2$-plane $E$ a regular $n$-gon as above, circumscribed in the circle of radius $r$ {around} the origin $o$. Denote its vertices by $v_1, \ldots, v_n$. Choose $\eps > 0$ small enough so that $r + 3 \eps < s$, and let $B_i = B^d(v_i, \eps)$ be the $d$-dimensional hyperbolic ball around $v_i$ of radius $\eps$. By construction, for any $i \neq j$ and any $x_i \in B_i$ and $x_j \in B_j$ one has
		\begin{equation*}
		d_\kk(o,x_i) < d_\kk(x_i,x_j).
		\end{equation*}
		Now consider the event $A_n$ that inside each ball $B_i$ lies precisely one point of the {Poisson process} $\eta$, and there are no other points inside the ball $B^d(o,r+\eps)$. Conditionally on $A_n$, one immediately sees that the origin is the radial nearest neighbour of each of the points lying in the balls $B_i$, and in particular, $\deg(o) \geq n$. Therefore,
		\begin{equation*}
		\PP(\deg(o) \geq n) \geq \PP (A_n) \geq \PP(\eta(B_1)=1)^n \cdot\PP\Big(\eta \Big(B^d(o,r+\eps)\setminus\bigcup_{i=1}^n B_i\Big) = 0\Big) > 0
		\end{equation*}
		and the proof is complete.
	\end{proof}

\section{Expectation asymptotics for $\cL_{R,\gamma,\kk}^{(\alpha)}$}\label{sec:exp_LR}
	
In this section we study the asymptotics of the expectations of the edge-length functionals of the radial spanning tree. Recall that we denote by $\ell_\kk(x, \eta)$ the distance from $x \in \eta$ to its radial nearest neighbour, and by 
		\begin{align*}
			\mathcal{L}_{R,\gamma,\kk}^{(\alpha)}:=\sum_{x\in\eta_R}\ell_\kk(x,{\eta})^\alpha
		\end{align*}
		the sum of the $\alpha$-th powers of the edge-lengths of $\RST(\eta_R)$. We begin with the following geometric result.
	
	\begin{lemma}\label{lem:Intersection}
Let $u>0$ be fixed. For $R>0$ let $B^d(o,R)$ be a hyperbolic ball around the origin $o \in \H^d(\kk)$, and let $B^d(x_R,u)$ be another hyperbolic ball of radius $u>0$ around a point $x_R$ lying on the boundary of the first ball, i.e.\ $d_\kk(o,x_R)=R$. Denote by ${\rm HB}_e(o)$ the horoball around some fixed ideal point $e\in\partial\mathbb{H}^d(\kk)$ passing through $o$. Then
			\begin{equation*}
		\lim_{R \to \infty} \mathcal{H}_\kk^d \big(  B^d(o,R) \cap B^d(x_R,u)\big) {= \cH^d_\kk({\rm HB}_e(o)\cap B^d(o,u))} = \kk^{-d/2}G(\skk u)
		\end{equation*}
		with the function {$G$} as in Theorem \ref{thm:ExpecationVariance} (a).
	\end{lemma}
\begin{proof}
Applying appropriate hyperbolic isometries, we may assume instead that the centre of {the ball with radius $u$} is fixed at the origin, and compute the volume of the intersection of the ball $B^d({o},u)$ with a sequence of {growing} balls $B_R$ of radii $R > 0$ passing through {$o$} with {centres tending to $e$.} As $R \to \infty$ the balls $B_R$ converge to a horoball ${\rm HB}_e(o)$. We conclude that
\begin{equation*}
\lim_{R \to \infty} \mathcal{H}_\kk^d \big(  B^d(o,R) \cap B^d(x_R,u)\big) {= \mathcal{H}_\kk^d \big( B^d(o,u) \cap {\rm HB}_e(o)\big).}
\end{equation*}
Finally, to compute the latter volume we consider the family $(H_t)_{t \in \R}$ of parallel horospheres corresponding to the horoball {${\rm HB}_e(o)$,} parametrised so that $H_t$ has signed distance $t$ from {$o$.} With this convention, the horoball {${\rm HB}_e(o)$} is the union $\bigcup_{t \geq 0}H_t$. Applying Lemma \ref{lem:disintegration_horo}, we get
\begin{equation}\label{eq:inter_vol_integ}
\mathcal{H}_\kk^d \big( B^d(o,u) \cap {\rm HB}_e(o)\big) = \int_{0}^\infty \cH_\kk^{d-1}(B^d({o},u) \cap H_t) \,\dint t.
\end{equation}
The volume appearing on the {right-hand} side is computed as in \cite[Proposition 4.1]{KRT}, which deals with the curvature $\kk=1$. It is zero unless $t \leq u$, in which case it is given by
\begin{equation*}
\cH^{d-1}_\kk(B^d({o},u) \cap H_t) = \kappa_{d-1}  \kk^{-{d-1 \over 2}} \left[2e^{-\skk t} \big(\cosh (\skk u) - \cosh (\skk t) \big)\right]^{d-1 \over 2}.
\end{equation*}
Substituting this into \eqref{eq:inter_vol_integ} and changing variables yields the result.
\end{proof}

The following lemma is essential for deriving the asymptotics of the expectation and provides moment estimates that are used for the proofs of the variance asymptotics and the central limit theorem as well.

{\begin{lemma}\label{lem:expectation_ell}
Let $\alpha> 0$.
\begin{itemize}
\item [(a)] For every $x\in\H^d(\kk)$,
$$
\E[\ell_\kk(x,\eta+\delta_x)^\alpha] = \alpha \int_0^{d_\kk(o,x)} u^{\alpha-1} \exp(-\gamma \mathcal{H}_\kk^d(B^d(x,u)\cap B^d(o,d_\kk(o,x)))) \;\mathrm{d}u.
$$ 
\item [(b)] There exists a constant $c(\alpha)>0$ only depending on $\alpha$ and $d$ such that
$$
\gamma^{\alpha/d}\E[\ell_\kk(x,\eta+\delta_x)^\alpha \id\{\ell_\kk(x,\eta+\delta_x)\geq t\}] \leq c(\alpha) \id\{ t\leq d_\kk(o,x) \} \exp(-\gamma V_\kk(t/2)/2)
$$
for all $x\in\H^d(\kk)$ and $t\geq 0$.
\end{itemize}
\end{lemma}
\begin{proof}
First we compute the expectation from part (b), which can be written as
$$
\E[\ell_\kk(x,\eta+\delta_x)^\alpha \id\{\ell_\kk(x,\eta+\delta_x)\geq t\}] = \int_0^\infty \p(\ell_\kk(x,\eta+\delta_x) \geq \max\{s^{1/\alpha},t\})\;\mathrm{d}s.
$$
{Since $\ell_\kk(x,\eta+\delta_x) \leq  d_\kk(o,x)$ and $\ell_\kk(x,\eta+\delta_x)\geq \tilde{u}$ if and only if} $\eta$ has no points in $B^d(x,\tilde{u})\cap B^d(o,d_\kk(o,x))$ for $\tilde{u}\in (0,d_\kk(o,x))$, we obtain
\begin{align*}
& \E[\ell_\kk(x,\eta+\delta_x)^\alpha \id\{\ell_\kk(x,\eta+\delta_x)\geq t\}] \\
& = \id\{ t\leq d_\kk(o,x) \} \int_0^{d_\kk(o,x)^\alpha} \p(\ell_\kk(x,\eta+\delta_x) \geq \max\{s^{1/\alpha},t\})\;\mathrm{d}s \\
& = \id\{ t\leq d_\kk(o,x) \} \int_0^{d_\kk(o,x)^\alpha} \p( \eta(B^d(x,\max\{s^{1/\alpha},t\})\cap B^d(o,d_\kk(o,x)))=0) \;\mathrm{d}s \\
& = \id\{ t\leq d_\kk(o,x) \} \int_0^{d_\kk(o,x)^\alpha} \exp(-\gamma \mathcal{H}_\kk^d(B^d(x,\max\{s^{1/\alpha},t\})\cap B^d(o,d_\kk(o,x)))) \;\mathrm{d}s \\
& = \id\{ t\leq d_\kk(o,x) \} \alpha\int_0^{d_\kk(o,x)} u^{\alpha-1} \exp(-\gamma \mathcal{H}_\kk^d(B^d(x,\max\{u,t\})\cap B^d(o,d_\kk(o,x)))) \;\mathrm{d}u,
\end{align*}
where we used the substitution $s = u^\alpha$ in the last step. For $t=0$ this yields part (a). Note that for $\tilde{u}\in(0,d_\kk(o,x))$, the intersection $B^d(x,\tilde{u})\cap B^d(o,d_\kk(o,x))$ contains the hyperbolic ball of radius $\tilde{u}/2$ around the midpoint between $x$ and the point on the boundary of $B^d(x,\tilde{u})$ closest to the origin. Hence, we have
\begin{equation}\label{eqn:intersection_balls}
\mathcal{H}_\kk^d(B^d(x,\tilde{u})\cap B^d(o,d_\kk(o,x))) \geq V_\kk(\tilde{u}/2).
\end{equation}
This implies
\begin{align*}
& \E[\ell_\kk(x,\eta+\delta_x)^\alpha \id\{\ell_\kk(x,\eta+\delta_x)\geq t\}] \\
& \leq \id\{ t\leq d_\kk(o,x) \} \alpha \int_0^{d_\kk(o,x)} u^{\alpha-1} \exp(-\gamma V_\kk(\max\{u,t\}/2)) \;\mathrm{d}u \\
& \leq \id\{ t\leq d_\kk(o,x) \} \exp(-\gamma V_\kk(t/2)/2) \alpha \int_0^\infty u^{\alpha-1} \exp(-\gamma V_\kk(u/2)/2) \;\mathrm{d}u.
\end{align*}
Since $\sinh(s)\geq s$ for all $s\geq 0$, it follows from \eqref{eq:volume_ball} that
$$
V_\kk(r) = \omega_d \kk^{-\frac{d-1}{2}} \int_0^r \sinh^{d-1}(\sqrt{\kk}v) \;\mathrm{d}v \geq d\kappa_d \kk^{-\frac{d-1}{2}} \int_0^r (\sqrt{\kk}v)^{d-1} \;\mathrm{d}v = \kappa_d r^d
$$
for all $r\geq 0$. This leads to
\begin{align*}
\gamma^{\alpha/d} \alpha \int_0^\infty u^{\alpha-1} \exp(-\gamma V_\kk(u/2)/2) \;\mathrm{d}u & \leq \gamma^{\alpha/d} \alpha \int_0^\infty u^{\alpha-1} \exp(-\gamma \kappa_d 2^{-d-1}u^d) \;\mathrm{d}u \\
& = \alpha \int_0^\infty v^{\alpha-1} \exp(- \kappa_d 2^{-d-1}v^d) \;\mathrm{d}v,
\end{align*}
which shows part (b).
\end{proof}}

With these preparations, we can prove the asymptotics for the expectations of the edge-length functionals.

\begin{proof}[Proof of Theorem \ref{thm:ExpecationVariance} (a)]
	The Mecke equation \eqref{eq:Mecke} for {Poisson processes} and Lemma \ref{lem:expectation_ell} (a) imply that
	\begin{align*}
		\E[\mathcal{L}_{R,\gamma,\kk}^{(\alpha)}]&=\E\left[\sum_{x\in\eta_R}\ell_\kk(x,\eta)^\alpha\right] =\gamma\int_{B^d(o,R)}\E[\ell_\kk(x,\eta+\delta_x)^\alpha]\;\mathcal{H}_\kk^d(\mathrm{d}x)\\
		&= \gamma\alpha\int_{B^d(o,R)} \int_0^{d_\kk(o,x)} u^{\alpha-1} \exp(-\gamma \mathcal{H}_\kk^d(B^d(x,u)\cap B^d(o,d_\kk(o,x)))) \;\mathrm{d}u \; \mathcal{H}_\kk^d(\mathrm{d}x).
	\end{align*}
Using the polar integration formula \eqref{eq:polar_integration} we compute
\begin{align*}
	&\E[\mathcal{L}_{R,\gamma,\kk}^{(\alpha)}]\\
	& = \gamma\alpha\omega_{d}\kk^{-{d-1\over 2}}\int_0^{R}\int_0^{s}u^{\alpha-1}\exp\left(-\gamma\mathcal{H}_\kk^d(B^d(x_s,u)\cap B^d(o,s))\right)\sinh^{d-1}(\skk s)\;\mathrm{d}u\,\mathrm{d}s,
\end{align*}
where $x_s$ is a point in $B^d(o,s)$ with $d_\kk(x_s,o)=s$. {Since $\E[\mathcal{L}_{R,\gamma,\kk}^{(\alpha)}]\to\infty$ and $V_\kk(R)\to\infty$ as $R\to\infty$, by l'Hospital's rule and \eqref{eq:volume_ball} we obtain
\begin{align*}
\lim\limits_{R\to\infty} \frac{\E[\mathcal{L}_{R,\gamma,\kk}^{(\alpha)}]}{V_\kk(R)} & = \lim\limits_{R\to\infty} \frac{\gamma\alpha\omega_{d}\kk^{-{d-1\over 2}}}{\omega_d\kk^{-\frac{d-1}{2}}\sinh^{d-1}(\sqrt{\kk}R)} \\
& \hspace{1.2cm} \times \int_0^{R}u^{\alpha-1}\exp\left(-\gamma\mathcal{H}_\kk^d(B^d(x_R,u)\cap B^d(o,R))\right)\sinh^{d-1}(\skk R)\;\mathrm{d}u \\
& = \lim\limits_{R\to\infty} \gamma\alpha \int_0^{R}u^{\alpha-1}\exp\left(-\gamma\mathcal{H}_\kk^d(B^d(x_R,u)\cap B^d(o,R))\right)\;\mathrm{d}u.
\end{align*}
Since, by \eqref{eqn:intersection_balls}, the integrand is bounded by the integrable function $u\mapsto u^{\alpha-1}\exp(-\gamma V_\kk(u/2))$, $u>0$, the dominated convergence theorem and Lemma \ref{lem:Intersection} lead to
\begin{align*}
\lim\limits_{R\to\infty} \frac{\E[\mathcal{L}_{R,\gamma,\kk}^{(\alpha)}]}{V_\kk(R)} & = \gamma\alpha \int_0^{\infty} \lim\limits_{R\to\infty} u^{\alpha-1}\exp\left(-\gamma\mathcal{H}_\kk^d(B^d(x_R,u)\cap B^d(o,R))\right)\;\mathrm{d}u \\
& = \gamma\alpha \int_0^{\infty} u^{\alpha-1} \exp\left(- \gamma \kk^{-d/2} G(\skk u)\right) \;\mathrm{d}u.
\end{align*}
Combining this with the change of variables ${v} =u^\alpha$ we conclude that
$$
\lim\limits_{R\to\infty}\frac{\E[\mathcal{L}_{R,\gamma,\kk}^{(\alpha)}]}{V_\kk(R) } =\gamma \int_0^\infty \exp\left(- \gamma \kk^{-d/2} G(\skk {v}^{1/ \alpha})\right)\,\dint {v}, 
$$
which completes the proof.}
\end{proof}

\section{A scaling property for $\cL_{R,\gamma,\kk}^{(\alpha)}$} \label{sec:scaling}

The following proposition provides for the distributional behaviour of the edge-length functional $\mathcal{L}_{R,\gamma,\kk}^{(\alpha)}$ a relation between the intensity $\gamma$ of the Poisson process $\eta$ and the curvature parameter $\kk$, which will allow us to deduce results for $\mathcal{L}_{R,\gamma,\kk}^{(\alpha)}$ with general $\kk$ from those for the special case $\kk=1$. In particular, it shows that increasing the intensity of the Poisson process is equivalent to simultaneously increasing the window size and shrinking the curvature -- or vice versa.

\begin{prop}\label{prop:Scaling}
Fix $\kk, R,\gamma,\alpha>0$, and recall that $\cL^{(\alpha)}_{R,\gamma,{\kk}}$ denotes the edge-length functional of the radial spanning tree based on a stationary Poisson process in $\H^d(\kk)$ with intensity $\gamma$. Then for all $\lambda > 0 $ one has the  following equality in law:
\[
\cL^{(\alpha)}_{R,\gamma,{\kk}} \overset{d}{=} \lambda^\alpha \cL^{(\alpha)}_{\lambda^{-1}R,\lambda^d\gamma,\lambda^2{\kk}}.
\]
\end{prop}

As a special case, the previous result shows that $\cL_{1,\gamma,\varkappa}^{(\alpha)}$ has the same distribution as $\gamma^{-\alpha/d}\cL_{\gamma^{1/d},1,\gamma^{-2/d}\varkappa}$, verifying the claim from the end of Section \ref{subsec:ComparisonEuclidean} about the asymptotic Euclidean behaviour of the {edge-length} functional for $\gamma\to \infty$.

\begin{proof}[Proof of Proposition \ref{prop:Scaling}]	
	The proof is based on the fact that the Riemannian metric $g_{\lambda^2 \kk}$ {is} isometric to the rescaling $\lambda^{-2}\, g_\kk$. More precisely, given any model $(\H^d(\kk),g_\kk)$ of the hyperbolic space of curvature $-\kk$, the same space with the rescaled metric $\lambda^{-2} g_\kk$ has constant sectional curvature $-\lambda^2\kk$, and is thus a model of {$\H^d(\lambda^2\kk)$.} Put differently, there is an origin-preserving diffeomorphism
	\[
	f : \H^d(\kk) \to \H^d(\lambda^2\kk) \qquad \text{such that} \qquad (f)^*g_{\lambda^2\kk}: = g_{\lambda^2\kk}(\dint f (\cdot), \dint f(\cdot) ) = \lambda^{-2}\, g_\kk
	\]
	(in the above terms, one can take as $f$ the identity map $(\H^d(\kk),g_\kk) \to (\H^d(\kk),\lambda^{-2} g_\kk)$). Here, $(f)^*g_{\lambda^2\kk}$ stands for the pullback of $g_{\lambda^2\kk}$ under $f$. In particular, the map $f$ rescales distances, i.e. $$ d_{\lambda^2 \kk}(f (x), f(y)) = { \lambda^{-1  }}\, d_\kk(x,y). $$ For balls in the corresponding spaces it follows immediately that 
	\[
	f (B_\kk^d(o,R)) = B_{\lambda^2\kk}^d (o,\lambda^{-1}R)
	\]
	and
	\[
	f_* \cH^d_\kk = \lambda^d \,\cH^d_{\lambda^2\kk},
	\]
	where $f_* \cH^d_\kk$ denotes the pushforward of $\cH^d_\kk$ under $f$. Let $\eta_{\gamma,\kk}$ be a stationary Poisson process on $\H^d(\kk)$ with intensity $\gamma>0$ and $\eta_{\lambda^d\gamma,\lambda^2\kk}$ be a stationary Poisson process on $\H^d({\lambda^2\kk})$ with intensity $\lambda^d\gamma$. Then $f(\eta_{\gamma,\kk}) \overset{d}{=} \eta_{\lambda^d\gamma,\lambda^2\kk}$. The same holds for their restrictions to balls of radius $R$ and $\lambda^{-1}R$:
	$
	f\left(\eta_{R,\gamma,\kk}\right) \overset{d}{=} \eta_{\lambda^{-1}R,\lambda^d\gamma,\lambda^2\kk}.
	$
	Moreover, since $f$ rescales distances and preserves the origin, it clearly preserves radial nearest neighbours, and hence for $x \in \eta_{\gamma,\kk}$ one has
	\[
	\ell_\kk(x,\eta_{\gamma,\kk}) = \lambda \, \ell_{\lambda^2\kk}\big(f(x),f(\eta_{\gamma,\kk})\big).
	\]
	Therefore,
	\begin{align*}
		\cL^{(\alpha)}_{R,\gamma,\kk} &= \sum_{x \in \eta_{R,\gamma,\kk}} \ell_\kk(x,\eta_{\gamma,\kk})^\alpha \overset{d}{=} \sum_{y \in \eta_{\lambda^{-1}R,\lambda^d\gamma,\lambda^2\kk}}  \left[\lambda \, \ell_{\lambda^2\kk} (y, \eta_{\lambda^d\gamma,\lambda^2\kk})\right]^\alpha 
		= \lambda^\alpha {\cL^{(\alpha)}_{\lambda^{-1}R,\lambda^d\gamma,\lambda^2\kk},}
	\end{align*}
	which completes the proof.
\end{proof}

\section{Variance asymptotics for $\cL_{R,\gamma,\kk}^{(\alpha)}$}\label{sec:variance}

In this section we consider the variance asymptotics of the edge-length functionals. It follows from the scaling relation in Proposition \ref{prop:Scaling} with $\lambda=1/\sqrt{\kk}$ and \eqref{eq:Vk-rescale} that
\begin{equation}\label{eqn:variance_kk}
\frac{\Var[\cL^{(\alpha)}_{R,\gamma,\kk}]}{V_\kk(R)} = \frac{\kk^{-\alpha} \Var[ \cL^{(\alpha)}_{\sqrt{\kk}R,\kk^{-d/2}\gamma,1} ]}{\kk^{-d/2} V_1(\sqrt{\kk} R)} = \kk^{d/2-\alpha} \frac{\Var[ \cL^{(\alpha)}_{\sqrt{\kk}R,\kk^{-d/2}\gamma,1} ]}{V_1(\sqrt{\kk} R)}.
\end{equation}
This identity implies for the asymptotic variance constants
\begin{equation}\label{eqn:asymptotic_variance_kk}
\VV^{(\alpha)}_{\gamma,\kk} = \kk^{\frac{d}{2}-\alpha} \VV^{(\alpha)}_{\kk^{-d/2}\gamma,1}
\end{equation}
so that $\gamma^{\frac{2\alpha}{d}-1}\VV^{(\alpha)}_{\gamma,\kk}= (\kk^{-d/2}\gamma)^{\frac{2\alpha}{d}-1}\VV^{(\alpha)}_{\kk^{-d/2}\gamma,1}$. Due to \eqref{eqn:variance_kk} and \eqref{eqn:asymptotic_variance_kk}, it is sufficient to prove Theorem \ref{thm:ExpecationVariance} (b) for the case $\kk=1$, which we consider in the following. We denote by $d_1$, $\cH^d_1$ and $V_1$ the distance function, the Hausdorff measure and the volume growth function of the hyperbolic space $\H^d:=\H^d(1)$ of curvature $-1$. Moreover, we write {$\eta$  and $\eta_R$} for the Poisson process with intensity $\gamma$ on $\H^d$ and its restriction to the ball of radius $R$ around $o$, and abbreviate similarly $\ell(\cdot, \eta)$  for the distance to the radial nearest neighbour. 

To formulate the precise variance {asymptotics} of the {edge-length functionals,} we need to consider a directed variant of the hyperbolic radial spanning {tree,} the so-called \emph{hyperbolic directed spanning forest}, which we describe next. This model was introduced in the Euclidean case in \cite{BB}, and was later generalised to hyperbolic space in \cite{Flammant}, see also \cite{CFT1} for its connection to the hyperbolic radial spanning tree. To define it, we fix an ideal point $e \in\partial \H^d$, and for a point {$y \in\H^d$ denote by $\HB_e(y)$} the horoball around $e$ with boundary passing through $y$. The directed spanning forest {$\DSF_e(\eta)$} based on the Poisson process {$\eta$} with respect to the ideal point $e$ {has the points of $\eta$ as vertices} and each vertex $y$ is connected to its (a.s.\ unique) nearest neighbour in the horoball $\HB_e(y)$. We denote by $\ell_e(y)$ the distance between $y$ and its nearest neighbour in {$\DSF_e(\eta + \delta_y)$,} and by $\ell_e^{(x)}(y)$ the distance from $y$ to its nearest neighbour in {$\DSF_e(\eta + \delta_y+\delta_x)$ for $x\in \H^d$.} With this notation the asymptotic variance constant takes the following form, {which proves the first half of Theorem \ref{thm:ExpecationVariance} (b)}.

\begin{prop}\label{prop:var_limit}
For $\alpha>0$ consider the edge-length functional $\mathcal{L}_{R,\gamma,1}^{(\alpha)}$ of the radial spanning tree $\RST(\eta)$ based on a stationary Poisson process $\eta$ on $\H^d$ with intensity $\gamma>0$.	Then the limit
\begin{align*}
	\VV_{\gamma,1}^{(\alpha)} & {:=} \lim\limits_{R\to\infty}\frac{\Var[\mathcal{L}_{R,\gamma,1}^{(\alpha)}]}{{V_1(R)}}
\end{align*}
exists and is equal to
\begin{align*}
& \gamma\E[\ell_{e}(o)^{2\alpha}] + 2\gamma^2 \int_{{\rm HB}_e(o)} \E[\ell_{e}^{(y)}(o)^{\alpha}\ell_{e}^{(o)}(y)^{\alpha}] - \E[\ell_{e}(o)^{\alpha}]\E[\ell_{e}(y)^{\alpha}] \, \mathcal{H}_1^d(\mathrm{d}y)
\end{align*}	
where $e \in \partial \H^d$ is an arbitrary ideal point.
\end{prop}

For the Euclidean case the asymptotic variance constant was derived in \cite[Lemma 3.4]{ST}. By rewriting in that formula the integral over $\mathbb{R}^d$ as two times the integral over a half-space with the origin in its boundary, one obtains the same expression as in Proposition \ref{prop:var_limit} with an integral over a half-space instead of over a horoball. This is due to the fact that in both the proof of Proposition \ref{prop:var_limit} below and the proof for the Euclidean case in \cite{ST} one has to deal with increasing balls with centres tending to infinity, which converge to a horoball in the hyperbolic case and a half-space in the Euclidean case. In this context, for the variance asymptotics of geometric functionals associated with Boolean models stronger differences between Euclidean and hyperbolic space were observed in \cite{HLS24}. These are caused by boundary effects, which are negligible in Euclidean case but become significant in hyperbolic space. Since we do not deal with edges of the radial spanning tree crossing the boundary of the observation window, such effects are not present in the current paper. However, we expect similar phenomena to occur if we take such edges into account.

We begin by introducing some further notation. First, for a point $p \in \H^d$ we consider the radial spanning tree $\RST_p(\eta)$ build on $\eta$ with base point $p$, i.e.\ $p$ plays now the role of the origin and becomes the root of $\RST_p(\eta)$. We denote by $\ell_p(x)$ the distance from {$x \in \mathbb{H}^d$} to its radial nearest neighbour {(with respect to $p$)} in {$\RST_p(\eta+\delta_x)$}, and by $\ell_p^{(y)}(x)$ the same distance in {$\RST_p(\eta+\delta_x+\delta_y)$ with $y \in \H^d$.} Observe that by the isometry-invariance of $\eta$, for any hyperbolic isometry $\varrho$ and any $x,y \in \H^d$ one has the equalities in law
\begin{equation}\label{eq:ell_invariance}
\ell(x)\overset{d}{=}\ell_{\varrho(o)}(\varrho(x))\qquad\text{and}\qquad
\big(\ell^{(y)}(x),\ell^{(x)}(y)\big) \overset{d}{=} \big( \ell_{{\varrho} (o)}^{(\varrho(y))}({\varrho} (x)) ,\ell_{{\varrho} (o)}^{({\varrho} (x))}({\varrho} (y))\big).
\end{equation}
{For $p,x,y\in\H^d$, the inequality $\ell_{p}^{(y)}(x)\leq \ell_{p}(x)$ is obvious since adding the point $y$ can only decrease the distance from $x$ to its radial nearest neighbour with respect to $p$. Thus, it follows from Lemma \ref{lem:expectation_ell} (b) that
\begin{equation}\label{eqn:uniformly_bounded_moments}
\sup_{p,x\in\H^d} \E[\ell_p(x)^\alpha] <\infty \quad \text{and} \quad \sup_{p,x,y\in\H^d} \E[\ell_p^{(y)}(x)^\alpha] <\infty
\end{equation}
for all $\alpha>0$.} 
Moreover, if ${\tau}:[0,\infty)\to\H^d$ is a geodesic ray tending towards the ideal point $e = {\tau}({\infty}) \in \partial\H^d$, it holds for $x,y\in\H^d$ that
\begin{equation}\label{eq:ell_limit}
\ell_{{\tau}(t)}(x) \xrightarrow[t\to\infty]{}\ell_e(x) \quad \text{and}\quad \ell_{{\tau}(t)}^{(y)}(x) \xrightarrow[t\to\infty]{}\ell_e^{(y)} (x) \quad \text{a.s.}
\end{equation}
We will need the following technical result.

\begin{lemma}\label{lem:majorant} \label{lemma:var_asymptotic}
For every $\alpha>0$ there exists a constant $C(\alpha)>0$ only depending on $\alpha$ and $d$ such that for all $p,z_1,z_2\in \H^d$,
	\begin{align*}
	&\big\lvert \mathbb{E}[\ell_{p}^{(z_2)}(z_1)^{\alpha}\ell_{p}^{(z_1)}(z_2)^{\alpha}]-\mathbb{E}[\ell_{p}(z_1)^{\alpha}]\mathbb{E}[\ell_{p}(z_2)^{\alpha}]\big\rvert\leq C(\alpha) \exp\Big(-{\gamma\over 2}V_1\Big({d_1(z_1,z_2)\over 4}\Big) \Big).
	\end{align*}
\end{lemma}

We postpone the proof of Lemma \ref{lem:majorant}, and first present the proof of Proposition \ref{prop:var_limit}.

\begin{proof}[Proof of Proposition \ref{prop:var_limit}]
	Due to the multivariate Mecke equation \eqref{eq:Meckemulti}, the variance is given by
	\begin{align*}
	\Var[\mathcal{L}_{R,\gamma,1}^{(\alpha)}]&=\E\bigg[\sum_{(x,y)\in\eta_{R,\neq}^2}\ell(x,{\eta})^{\alpha}\ell(y,{\eta})^{\alpha}\bigg]-\left(\E\bigg[\sum_{x\in\eta_R}\ell(x,{\eta})^{\alpha}\bigg]\right)^2+\E\bigg[\sum_{x\in\eta_R}\ell(x,{\eta})^{2\alpha}\bigg]\\
	&= T(R) + \E [\cL_{R,\gamma,1}^{(2\alpha)}],
	\end{align*}
	where
	\begin{align*}
		T(R) & := \gamma^2\int_{B^d(o,R)}\int_{B^d(o,R)} \E[\ell^{(y)}(x)^{\alpha}\ell^{(x)}(y)^{\alpha}]-\E[\ell(x)^{\alpha}]\E[\ell(y)^{\alpha}] \;{\mathcal{H}^d_1}(\mathrm{d}y)\;{\mathcal{H}_1^d}(\mathrm{d}x) .
	\end{align*}
	By Theorem \ref{thm:ExpecationVariance} (a) we have
	\begin{align*}
	\lim\limits_{R\to\infty}\frac{\E [\cL_{R,\gamma,1}^{(2\alpha)}]}{{V_1(R)}}=\gamma\int_0^\infty e^{-\gamma {G(u^{1/(2\alpha)})}}\,\mathrm{d}u.
	\end{align*}
	On the other hand, a computation as in the proof of Lemma \ref{lem:expectation_ell} (a) shows
	\begin{equation*}
	\E [\ell_e(o)^{2\alpha}] = \int_{0}^{\infty} \PP (\ell_e(o) \geq u^{1/(2\alpha)}) \,\dint u = \int_{0}^\infty \exp \big(-\gamma {\cH^{d}_1}(\HB_e(o) \cap {B^d}(o,u^{1/(2\alpha)})) \big) \,\dint u.
	\end{equation*}
	Finally, Lemma \ref{lem:Intersection} implies that {$\cH^{d}_1(\HB_e(o) \cap B^d(o,u^{1/(2\alpha)})) = G(u^{1/(2\alpha)})$} and hence
	\begin{equation}\label{eq:lim_T2}
	\lim\limits_{R\to\infty}\frac{\E [\cL_{R,\gamma,1}^{(2\alpha)}]}{{V_1(R)}} = \gamma \E [\ell_e(o)^{2\alpha}].
	\end{equation}
	For the term $T(R)$ we write, by symmetry and the polar integration formula \eqref{eq:polar_integration},
\begin{align*}
T(R) & = 2 \gamma^2 \int_{B^d(o,R)} \int_{B^d(o,{d_1(o,x)})} \E[\ell^{(y)}(x)^{\alpha}\ell^{(x)}(y)^{\alpha}]-\E[\ell(x)^{\alpha}]\E[\ell(y)^{\alpha}] \;{\mathcal{H}_1^d}(\mathrm{d}y)\;{\mathcal{H}^d_1}(\mathrm{d}x) \\
& = 2 \gamma^2 \omega_d \int_0^R \sinh^{d-1}(s) \int_{B^d(o,s)} \E[\ell^{(y)}(x_s)^{\alpha}\ell^{(x_s)}(y)^{\alpha}]-\E[\ell(x_s)^{\alpha}]\E[\ell(y)^{\alpha}] \;{\mathcal{H}_1^d}(\mathrm{d}y) \;\mathrm{d}s,
\end{align*}	
where $x_s$ is the point on the geodesic ray in $\mathbb{H}^d$ from $o$ tending towards the ideal point $e\in\partial\mathbb{H}^d$ such that $d_1(x_s,o)=s$ for $s\in[0,\infty)$. From \eqref{eq:volume_ball} and l'Hospital's rule we deduce
\begin{align*}
& \lim_{R\to\infty} \frac{T(R)}{V_1(R)} \\
&	= \lim_{R\to\infty} \frac{2 \gamma^2 \omega_d \sinh^{d-1}(R)}{\omega_d \sinh^{d-1}(R)} \int_{B^d(o,R)} \E[\ell^{(y)}(x_R)^{\alpha}\ell^{(x_R)}(y)^{\alpha}]-\E[\ell(x_R)^{\alpha}]\E[\ell(y)^{\alpha}] \; \mathcal{H}_1^d(\mathrm{d}y) \\
&	= \lim_{R\to\infty} 2 \gamma^2 \int_{B^d(o,R)} \E[\ell^{(y)}(x_R)^{\alpha}\ell^{(x_R)}(y)^{\alpha}]-\E[\ell(x_R)^{\alpha}]\E[\ell(y)^{\alpha}] \; \mathcal{H}_1^d(\mathrm{d}y).
\end{align*}
Using \eqref{eq:ell_invariance} to change the roles of $x_R$ and $o$, this can be rewritten as
\begin{align}
&	\lim_{R\to\infty} 2 \gamma^2 \int_{B^d(x_R,R)} \E[\ell^{(y)}_{x_R}(o)^{\alpha}\ell^{(o)}_{x_R}(y)^{\alpha}]-\E[\ell_{x_R}(o)^{\alpha}]\E[\ell_{x_R}(y)^{\alpha}] \;\mathcal{H}_1^d(\mathrm{d}y) \notag \\
&	= \lim_{R\to\infty} 2 \gamma^2 \int_{\mathbb{H}^d} \mathbbm{1}\{y\in B^d(x_R,R)\} \big(\E[\ell^{(y)}_{x_R}(o)^{\alpha}\ell^{(o)}_{x_R}(y)^{\alpha}]-\E[\ell_{x_R}(o)^{\alpha}]\E[\ell_{x_R}(y)^{\alpha}]\big) \;\mathcal{H}_1^d(\mathrm{d}y). \label{eqn:limit_T_V}
\end{align}
By Lemma \ref{lem:majorant}, the absolute value of the integrand is bounded by $C(\alpha) \exp\big(-{\gamma\over 2}V_1\big({d_1(o,y)\over 4}\big) \big)$ for $y\in\mathbb{H}^d$. Since, by a further application of \eqref{eq:polar_integration},
$$
\int_{\H^d} \exp\Big(-{\gamma\over 2}V_1\Big({d_1(o,y)\over 4}\Big) \Big) {\;\cH^d_1}(\dint y) = \omega_d \int_{0}^{\infty} \sinh^{d-1}(s) e^{-\gamma V_1({s/4})/2}\,\dint s < \infty,
$$
we may apply the dominated convergence theorem in \eqref{eqn:limit_T_V}. {Combining \eqref{eqn:uniformly_bounded_moments} with \eqref{eq:ell_limit}, we deduce, for $y\in\mathbb{H}^d$,} 
	\begin{align*}
	\E[\ell_{x_R}(o)^{\alpha}]\E[\ell_{x_R}(y)^{\alpha}] &\xrightarrow[R\to\infty]{} \E[\ell_{e}(o)^{\alpha}]\E[\ell_{e}(y)^{\alpha}] \\ \intertext{and}
    \E[\ell_{x_R}^{(z)}(o)^{\alpha}\ell_{x_R}^{(o)}(y)^{\alpha}]&\xrightarrow[R\to\infty]{} \E[\ell_{e}^{(y)}(o)^{\alpha}\ell_{e}^{(o)}(z)^{\alpha}].
    \end{align*}
Moreover, we have
$$
\mathbbm{1}\{y\in B^d(x_R,R)\} \xrightarrow[R\to\infty]{} \mathbbm{1}\{y\in {\rm HB}_e(o)\} 
$$
for $\mathcal{H}_1^d$-a.e.\ $y\in\mathbb{H}^d$. Altogether, we obtain
$$
\lim_{R\to\infty} \frac{T(R)}{V_1(R)} = 2\gamma^2 \int_{{\rm HB}_e(o)} \E[\ell_{e}^{(y)}(o)^{\alpha}\ell_{e}^{(o)}(y)^{\alpha}] - \E[\ell_{e}(o)^{\alpha}]\E[\ell_{e}(y)^{\alpha}] \, \mathcal{H}_1^d(\mathrm{d}y).
$$
Combining this with \eqref{eq:lim_T2} completes the proof.
\end{proof}

Finally, we return to the proof of the technical Lemma \ref{lem:majorant}.

\begin{proof}[Proof of Lemma \ref{lem:majorant}]
For $i\in\{1,2\}$ define the event
	\begin{align*}
		A_i:=\left\{\eta\left(B^d\left(\hat{z}_i,\frac{{d_1}(z_1,z_2)}{4}\right)\right)=0\right\}
	\end{align*}
	with $\hat{z}_i\in\mathbb{H}^d$	being the point on the geodesic ray from $z_i$ in the direction of $p$ such that $d_1(\hat{z}_i,z_i)=\frac{d_1(z_1,z_2)}{4}$. We note that in case of $A_1^c$, $\ell_p^{(z_2)}(z_1)$ is completely determined by the points of $\eta$ within $B^d(z_1,\frac{d_1(z_1,z_2)}{2})$ since there are potential radial nearest neighbours of $z_1$ in the latter ball. Indeed, for $p\notin B^d(\hat{z}_1,d_1(z_1,z_2)/4)$ the points of $\eta$ in $B^d(\hat{z}_1,d_1(z_1,z_2)/4)$ are closer to $p$ than $z_1$ and, thus, potential radial nearest neighbours of $z_1$, while for $p\in B^d(\hat{z}_1,d_1(z_1,z_2)/4)$, $p$ is a potential radial nearest neighbour of $z_1$. Analogously, $\ell_p^{(z_1)}(z_2)$ is in the case of $A_2^c$ completely determined by the points of $\eta$ in $B^d(z_2,\frac{d_1(z_1,z_2)}{2})$. Due to the independence of the events $A_1$ and $A_2$ by the defining property (ii) of a Poisson process we have for $A:=A_1\cup A_2$ that
	\begin{align*}
		&\E[\ell_p^{(z_2)}(z_1)^{\alpha}\ell_p^{(z_1)}(z_2)^{\alpha}\mathbbm{1}_{A^c}]=\E[\ell_p(z_1)^{\alpha}\mathbbm{1}_{A_1^c}]\E[\ell_p(z_2)^{\alpha}\mathbbm{1}_{A_2^c}].
	\end{align*}
	{It} follows together with the Cauchy-Schwarz inequality and {\eqref{eqn:uniformly_bounded_moments}} that
	\begin{align*}
		&\vert\E[\ell_p^{(z_2)}(z_1)^{\alpha}\ell_p^{(z_1)}(z_2)^{\alpha}]-\E[\ell_p(z_1)^{\alpha}]\E[\ell_p(z_2)^{\alpha}]\rvert\\
		&=\vert\E[\ell_p^{(z_2)}(z_1)^{\alpha}\ell_p^{(z_1)}(z_2)^{\alpha}(\mathbbm{1}_A+\mathbbm{1}_{A^c})]-\E[\ell_p(z_1)^{\alpha}(\mathbbm{1}_{A_1}+\mathbbm{1}_{A_1^c})]\E[\ell_p(z_2)^{\alpha}(\mathbbm{1}_{A_2}+\mathbbm{1}_{A_2^c})]\rvert\\
		&\leq \E[\ell_p^{(z_2)}(z_1)^{\alpha}\ell_p^{(z_1)}(z_2)^{\alpha}\mathbbm{1}_A]+\E[\ell_p(z_1)^{\alpha}\mathbbm{1}_{A_1}]\E[\ell_p(z_2)^{\alpha}\mathbbm{1}_{A_2}]\\
		&\quad+\E[\ell_p(z_1)^{\alpha}\mathbbm{1}_{A_1}]\E[\ell_p(z_2)^{\alpha}\mathbbm{1}_{A_2^c}]+\E[\ell_p(z_1)^{\alpha}\mathbbm{1}_{A_1^c}]\E[\ell_p(z_2)^{\alpha}\mathbbm{1}_{A_2}]\\
		&\leq \E[\ell_p^{(z_2)}(z_1)^{4\alpha}]^{1/4} \E[\ell_p^{(z_1)}(z_2)^{4\alpha}]^{1/4}\mathbb{P}(A)^{1/2}+\E[\ell_p(z_1)^{\alpha}\mathbbm{1}_{A_1}]\E[\ell_p(z_2)^{\alpha}]\\&\quad+\E[\ell_p(z_2)^{\alpha}\mathbbm{1}_{A_2}]\E[\ell_p(z_1)^{\alpha}]\\
		&\leq \E[\ell_p^{(z_2)}(z_1)^{4\alpha}]^{1/4} \E[\ell_p^{(z_1)}(z_2)^{4\alpha}]^{1/4}\mathbb{P}(A)^{1/2}+\E[\ell_p(z_1)^{2\alpha}]^{1/2}\E[\ell_p(z_2)^{\alpha}]\mathbb{P}(A_1)^{1/2}\\&\quad+\E[\ell_p(z_2)^{2\alpha}]^{1/2}\E[\ell_p(z_1)^{\alpha}]\mathbb{P}(A_2)^{1/2}\\
		&\leq C_1(\alpha)(\mathbb{P}(A)^{1/2}+\mathbb{P}(A_1)^{1/2}+\mathbb{P}(A_2)^{1/2})\\
		&\leq C_2(\alpha)(\mathbb{P}(A_1)^{1/2}+\mathbb{P}(A_2)^{1/2})\\
		&= 2C_2(\alpha)\exp\Big(-{\gamma\over 2}\mathcal{H}^d_1\Big(B^d\Big(o,\frac{d_1(z_1,z_2)}{4}\Big)\Big)\Big) = 2C_2(\alpha)\exp\Big(-{\gamma\over 2}V_1\Big({d_1(z_1,z_2)\over 4}\Big) \Big)
	\end{align*}
	for suitable constants $C_1(\alpha),C_2(\alpha)>0$ only depending on $\alpha$ and $d$, which completes the proof.
\end{proof}

\section{Central limit theorem and variance bounds for $\cL_{R,\gamma,\kk}^{(\alpha)}$}\label{sec:CLT}

In this section we establish the central limit theorem for $\mathcal{L}_{R,\gamma,\kk}^{(\alpha)}$ (see Theorem \ref{thm:CLT}) as well as the positivity of the asymptotic variance, proving the second statement in Theorem \ref{thm:ExpecationVariance} (b). We note that it suffices to consider the case of $\kk=1$ since, by the scaling relation in Proposition \ref{prop:Scaling} with $\lambda=1/\sqrt{\varkappa}$ and \eqref{eq:Vk-rescale},
$$
d_\diamondsuit\left(\frac{\cL_{R,\gamma,\kk}^{(\alpha)}-\E[\cL_{R,\gamma,\kk}^{(\alpha)}]}{\sqrt{\Var[\cL_{R,\gamma,\kk}^{(\alpha)}]}},N\right)=d_\diamondsuit\left(\frac{\cL_{\sqrt{\kk}R,\kk^{-d/2}\gamma,1}^{(\alpha)}-\E[\cL_{\sqrt{\kk}R,\kk^{-d/2}\gamma,1}^{(\alpha)}]}{\sqrt{\Var[\cL_{\sqrt{\kk}R,\kk^{-d/2}\gamma,1}^{(\alpha)}]}},N\right)
$$
for $\diamondsuit\in\{W,K\}$ and
$$
\sqrt{\kk^{-d/2}\gamma V_1(\sqrt{\kk}R)} = \sqrt{\gamma V_\kk(R)}.
$$
Therefore, we continue to restrict ourselves to this case, and write $d_1$, $\cH^d_1$, $V_1$, $\H^d$ and $\ell$ throughout. Observe, for the proof of the Central Limit Theorem \ref{thm:CLT}, that by the scaling relations given by Proposition \ref{prop:Scaling}, the assumption $\varkappa^{-d/2}\gamma\geq c_0$ reduces in this case to simply $\gamma \geq c_0$.

Let $\alpha>0$ and $c_0>0$ be fixed in the following. We first state four technical lemmas, whose proofs we defer to the end of this section. 
Lemma \ref{lem:bounded_moments} and Lemma \ref{lem:prop_seconddiff_neq_0} are derived analogously to Lemma 4.1 and Lemma 4.2 in \cite{ST} with the help of Lemma \ref{lem:help_integral}.

\begin{lemma}
	\label{lem:help_integral}
	For all $c_1,c_2>0$ there exist constants $\tilde{c}_1,\tilde{c}_2>0$ only depending on $c_1, c_2$, $d$ and $c_0$ such that for $\gamma\geq c_0$ and $s\geq0$,
	\begin{align*}
		\int_{\H^d\backslash B^d(o,s)}\gamma e^{-\gamma c_1{V_1}(\frac{{d_1}(x,o)}{c_2})}\;{\mathcal{H}^d_1}(\mathrm{d}x)\leq \tilde{c}_1e^{-\gamma \tilde{c}_2{V_1}(\frac{s}{c_2})}.
	\end{align*}
\end{lemma}

\begin{lemma}
	\label{lem:bounded_moments}
	For $p\in\mathbb{N}$ there exist constants $C_{1,p},C_{2,p}>0$ depending on $p$, $\alpha$, $d$ and $c_0$ such that for all $R>0$, $z,z_1,z_2\in B^d(o,R)$ and $\gamma\geq c_0$,
	\begin{align*}
		\gamma^{\alpha p/d}\E[\lvert D_z\mathcal{L}_{R,\gamma,1}^{(\alpha)}\rvert^p]\leq C_{1,p}\;\;\;\;\text{ and }\;\;\;\;\gamma^{\alpha p/d}\E[\lvert D_{z_1,z_2}^2\mathcal{L}_{R,\gamma,1}^{(\alpha)}\rvert^p]\leq C_{2,p}.
	\end{align*}
\end{lemma}

\begin{lemma}
	\label{lem:prop_seconddiff_neq_0}
	There are constants $c_1,c_2>0$ {depending on $\alpha$, $d$ and $c_0$} such that for all {$R>0$,} $z_1,z_2\in B^d(o, R)$ and $\gamma\geq c_0$,
	\begin{align*}
		\p(D_{z_1,z_2}^2\mathcal{L}_{R,\gamma,1}^{(\alpha)}\neq 0)\leq c_1e^{-\gamma c_2 {V_1(d_1}(z_1,z_2)/{4}) }.
	\end{align*}
\end{lemma}

To establish the positivity of the asymptotic variance we additionally need the following result, which shows that the first-order difference operator is non-zero with positive probability.

\begin{lemma}\label{lem:D_lower_bound}
There exist constants {$c,C>0$ depending on $\alpha$, $d$ and $c_0$ such that for all $\gamma\geq c_0$ and $R>3$ one has that}
\begin{equation*}
\PP (D_z \mathcal{L}_{R,\gamma,1}^{(\alpha)} \geq \gamma^{-\alpha/d}c) \geq C
\end{equation*}
{for all $z \in B^d(o,R) \setminus B^d(o,3)$}.
\end{lemma}

Assuming these results, we proceed with the proofs of the main results of this section. We first prove the boundedness statement in part (b) of Theorem \ref{thm:ExpecationVariance} by deriving the following bounds on the variance.

\begin{prop}\label{prop:var_bound}
There exist constants $c_\ell,c_u>0$ and $r_0>0$ depending on $d$, $\alpha$ and $c_0$ such that for all $\gamma\geq c_0$ and $R\geq r_0$,
\begin{align*}
c_\ell \gamma^{1-2\alpha/d}{V_1}(R) \leq \Var[\mathcal{L}_{R,\gamma,1}^{(\alpha)}] \leq c_u \gamma^{1-2\alpha/d}{V_1}(R).
\end{align*}
\end{prop}

\begin{proof}
	For the upper bound we use the Poincar\'e inequality \eqref{eq:Poincare}, which gives
	\begin{equation*}
	\Var[\mathcal{L}_{R,\gamma,1}^{(\alpha)}] \leq \gamma\, \int_{B^d(o,R)} \E \left[\left(D_z \mathcal{L}_{R,\gamma,1}^{(\alpha)}\right)^2\right] \, {\cH^{d}_1} (\dint z).
	\end{equation*}
	For the lower bound we employ the reverse Poincar\'e inequality \eqref{eq:reversePoincare}. 
In our case, assuming we can find some constant {$a \geq 0$} {such that}
	\begin{align*}
	 & \gamma^2 {\int_{B^d(o,R)} \int_{B^d(o,R)}} \E\left[ (D^2_{z_1, z_2} \mathcal{L}_{R,\gamma,1}^{(\alpha)})^2\right] \, {\cH^d_1}(\dint z_1) \, {\cH^d_1}(\dint z_2) \\
	& \leq  {a}  \gamma\int_{B^d(o,R)} \E \left[\left(D_z \mathcal{L}_{R,\gamma,1}^{(\alpha)}\right)^2  \right] {\cH^d_1}(\dint z) < {\infty},
	\end{align*}
	it implies that
	\begin{align*}
	\Var[\mathcal{L}_{R,\gamma,1}^{(\alpha)}] \geq \frac{4\gamma}{({a}+2)^2}  \int_{B^d(o,R)} \E \left[\left(D_z \mathcal{L}_{R,\gamma,1}^{(\alpha)}\right)^2 \right]\, {\cH^{d}_1} (\dint z).
	\end{align*}
Combining these two results, we see that our proof is complete once we find constants {$m,M>0$} for which
\begin{align}
	m\gamma^{1-2\alpha/d} \, {V_1}(R) \leq  \gamma\int_{B^d(o,R)} \E\left[\left(D_z\mathcal{L}_{R,\gamma,1}^{(\alpha)}\right)^2 \right] {\cH^d_1}(\dint z) &\leq M\gamma^{1-2\alpha/d}\, {V_1}(R) \label{eq:D-L2-bound} \\ \intertext{and}
	\gamma^2\int_{B^d(o,R)}\int_{B^d(o,R)} \E  \left[(D^2_{z_1,z_2} \mathcal{L}_{R,\gamma,1}^{(\alpha)})^2\right] \, {\cH^d_1}(\dint z_1) \, {\cH^d_1}(\dint z_2)  & \leq M\gamma^{1-2\alpha/d} \, {V_1}(R) \label{eq:D2-L2-bound}
	\end{align}
	are satisfied for $R$ sufficiently large. We start with proving \eqref{eq:D-L2-bound}. The upper bound follows immediately from Lemma \ref{lem:bounded_moments} {because}
	\begin{align*}
\gamma	\int_{B^d(o,R)} \E\left[\left( D_z\mathcal{L}_{R,\gamma,1}^{(\alpha)}\right)^2\right]\; {\mathcal{H}^d_1}(\mathrm{d}z) \leq \gamma \int_{B^d(o,R)} \gamma^{-2\alpha/d} C_{1,2} \, {\cH^d_1}(\dint z) = C_{1,2}\gamma^{1-2\alpha/d} {V_1}(R).
	\end{align*}
	For the lower bound, we use Lemma \ref{lem:D_lower_bound} to obtain
	\begin{align*}
		\gamma\int_{B^d(o,R)}\E\left[ (D_z\mathcal{L}_{R,\gamma,1}^{(\alpha)})^2\right]\; {\mathcal{H}^d_1}(\mathrm{d}z)&\geq \gamma\int_{B^d(o,R)\backslash B^d(o,3)} \gamma^{-2\alpha/d}c^2\p(D_z\mathcal{L}_{R,\gamma,1}^{(\alpha)}\geq  \gamma^{-\alpha/d}c)\; {\mathcal{H}^d_1}(\mathrm{d}z)\\
		&\geq \gamma^{1-2\alpha/d}c^2C  \big({V_1}(R)- {V_1}(3)\big) \\
		&\geq \gamma^{1-2\alpha/d}{c^2C \over 2}  {V_1}(R)
	\end{align*}
	for sufficiently large $R > 0$. This proves \eqref{eq:D-L2-bound}.
	
	We proceed with the proof of \eqref{eq:D2-L2-bound}. The Cauchy-Schwarz inequality together with Lemma \ref{lem:bounded_moments} implies
	\begin{align*}
		\E[\lvert D_{z_1,z_2}^2\mathcal{L}_{R,\gamma,1}^{(\alpha)}\rvert^2]&\leq \E[\lvert D_{z_1,z_2}^2\mathcal{L}_{R,\gamma,1}^{(\alpha)}\rvert^4]^{1/2}\p(D_{z_1,z_2}^2\mathcal{L}_{R,\gamma,1}^{(\alpha)}\neq 0)^{1/2}\\&\leq \gamma^{-2\alpha/d}C_{2,4}^{1/2}\p(D_{z_1,z_2}^2\mathcal{L}_{R,\gamma,1}^{(\alpha)}\neq 0)^{1/2}
	\end{align*}
	and in conjunction with Lemma \ref{lem:prop_seconddiff_neq_0} we have
	\begin{align*}
		& \gamma^2 {\int_{B^d(o,R)} \int_{B^d(o,R)}} \E\left[ \left(D_{z_1,z_2}^2\mathcal{L}_{R,\gamma,1}^{(\alpha)}\right)^2 \right]\; {\mathcal{H}^d_1}(\mathrm{d}z_1)\; {\mathcal{H}^d_1}(\mathrm{d}z_2)\\&\leq \gamma^{-2\alpha/d}C_{2,4}^{1/2}\gamma^2 {\int_{B^d(o,R)} \int_{B^d(o,R)}} \p(D_{z_1,z_2}^2\mathcal{L}_{R,\gamma,1}^{(\alpha)}\neq 0)^{1/2}\; {\mathcal{H}^d_1}(\mathrm{d}z_1)\; {\mathcal{H}^d_1}(\mathrm{d}z_2)\\
		&\leq \gamma^{1-2\alpha/d}{c_1^{1/2}} C_{2,4}^{1/2} {\int_{B^d(o,R)} \int_{B^d(o,R)}} \gamma e^{-\gamma c_2 {V_1(d_1}(z_1,z_2)/ 4)/2 }\; {\mathcal{H}^d_1}(\mathrm{d}z_1)\; {\mathcal{H}^d_1}(\mathrm{d}z_2)\\ 
		& \leq \gamma^{1-2\alpha/d}{c_1^{1/2}}C_{2,4}^{1/2} \int_{B^d(o,R)}  \left[\int_{\H^d }\gamma e^{-\gamma {c_2} {V_1(d_1}(z_1,z_2)/ 4){/2} }\, {\cH^d_1}(\dint z_2) \right] {\cH^d_1}(\dint z_1)\\
		& = \gamma^{1-2\alpha/d}{c_1^{1/2}}C_{2,4}^{1/2} \left[\int_{\H^d }\gamma e^{-\gamma {c_2} {V_1(d_1}(o,w)/ 4){/2} }\, {\cH^d_1}(\dint w) \right]  {V_1}(R) .
	\end{align*}
By Lemma \ref{lem:help_integral} the integral inside the brackets converges and is uniformly bounded for all $\gamma\geq c_0$. This completes the proof of \eqref{eq:D2-L2-bound} and with it, the proof of the proposition.
\end{proof}

Next we turn to the proof of the central limit theorem.

\begin{proof}[Proof of Theorem \ref{thm:CLT}]
We use Lemma \ref{lem:bounded_moments}, which shows that the moments of the first- and second-order difference operators of $\gamma^{\alpha/d}\mathcal{L}_{R,\gamma,1}^{(\alpha)}$ are uniformly bounded by some constant $\bar{c}\geq 1$ for all $\gamma\geq {c_0}$. Then, Theorem \ref{thm:normalapproximation_bounds_dW_dK} implies the inequality 
	\begin{align*}
			d_{{\diamondsuit}}\left( {\mathcal{L}_{R,\gamma,1}^{(\alpha)} - \E[\mathcal{L}_{R,\gamma,1}^{(\alpha)}]  \over \sqrt{\Var[\mathcal{L}_{R,\gamma,1}^{(\alpha)}]}}, N    \right)&= d_{{\diamondsuit}}\left( {\gamma^{\alpha/d}\mathcal{L}_{R,\gamma,1}^{(\alpha)} - \E[\gamma^{\alpha/d}\mathcal{L}_{R,\gamma,1}^{(\alpha)}]  \over \sqrt{\Var[\gamma^{\alpha/d}\mathcal{L}_{R,\gamma,1}^{(\alpha)}]}}, N    \right)\\&\leq \frac{\bar{c}I_1^{1/2}}{\Var[\gamma^{\alpha/d}\mathcal{L}_{R,\gamma,1}^{(\alpha)}]}+\frac{2\bar{c}I_1}{\Var[\gamma^{\alpha/d}\mathcal{L}_{R,\gamma,1}^{(\alpha)}]^{3/2}}+\frac{\bar{c}I_1^{5/4}+2\bar{c}I_1^{3/2}}{\Var[\gamma^{\alpha/d}\mathcal{L}_{R,\gamma,1}^{(\alpha)}]^2}\\
			& \quad	+\frac{5\bar{c}}{\Var[\gamma^{\alpha/d}\mathcal{L}_{R,\gamma,1}^{(\alpha)}]}I_2+\frac{\sqrt{6}\bar{c}+\sqrt{3}\bar{c}}{\Var[\gamma^{\alpha/d}\mathcal{L}_{R,\gamma,1}^{(\alpha)}]}I_3
	\end{align*}
{for $\diamondsuit=W$ and $\diamondsuit=K$,}	where $N$ is a standard Gaussian random variable and the terms $I_1$, $I_2$ and $I_3$ are defined by
	\begin{align*}
	I_1 & := \gamma \int_{B^d(o,R)} \PP(D_z\mathcal{L}_{R,\gamma,1}^{(\alpha)} \neq 0)^{1/10} \, {\cH^d_1}(\dint z), \\
	I_2 & := \Bigg(\gamma\int_{B^d(o,R)}\Bigg(\gamma\int_{B^d(o,R)}\p(D^2_{x_1,x_2}\mathcal{L}_{R,\gamma,1}^{(\alpha)}\neq 0)^{1/20}\; {\mathcal{H}^d_1}(\mathrm{d}x_2)\Bigg)^2\; {\mathcal{H}^d_1}(\mathrm{d}x_1)\Bigg)^{1/2}
	\intertext{and}
	I_3 &:= \Bigg(\gamma^2 {\int_{B^d(o,R)}} \int_{B^d(o,R)}\p(D_{x_1,x_2}^2\mathcal{L}_{R,\gamma,1}^{(\alpha)}\neq 0)^{1/10}\; {\mathcal{H}^d_1}(\mathrm{d}x_1)\; {\mathcal{H}^d_1}(\mathrm{d}x_2)\Bigg)^{1/2}.
	\end{align*}
	For $I_1$ we use the trivial bound $\PP(D_z\mathcal{L}_{R,\gamma,1}^{(\alpha)} \neq 0) \leq 1$ which gives $I_1 \leq \gamma {V_1}(R)$. To bound $I_2$ we note that Lemma \ref{lem:prop_seconddiff_neq_0} and Lemma \ref{lem:help_integral} yield
 	\begin{align*}
 	\gamma \int_{B^d(o,R)}\p(D_{x_1,x_2}^2\mathcal{L}_{R,\gamma,1}^{(\alpha)}\neq 0)^{1/20}\,\cH_1^d(\mathrm{d}x_1)\leq c_1^{1/20} \int_{\mathbb{H}^d}\gamma e^{-\gamma c_2 {V_1(d_1}(x_1,o)/{4}){/20} }\,\cH_1^d(\mathrm{d}x_1)\leq C_2
 	\end{align*}
 	for all $\gamma\geq c_0$ with some suitable constant $C_2>0$ and therefore
	\[
	I_2 \leq C_2\gamma^{1/2}  {V_1}(R)^{1/2}.
	\]
	Similarly, we have $I_3\leq C_3\gamma^{1/2}  {V_1}(R)^{1/2}$ for a suitable constant $C_3>0$. Together with the lower variance bound provided by Proposition \ref{prop:var_bound} we obtain for $R\geq r_0$,
	\begin{align*}
	d_\diamondsuit\left( {\mathcal{L}_{R,\gamma,1}^{(\alpha)} - \E[\mathcal{L}_{R,\gamma,1}^{(\alpha)}]  \over \sqrt{\Var[\mathcal{L}_{R,\gamma,1}^{(\alpha)}]}}, N    \right)  &\leq \frac{\bar{c}\gamma^{1/2} {V_1}(R)^{1/2}}{c_\ell\gamma {V_1}(R)}+\frac{2\bar{c}\gamma {V_1}(R)}{{c_\ell^{3/2}}\gamma^{3/2} {V_1}(R)^{3/2}} \\
	& \quad +\frac{\bar{c}\gamma^{5/4} {V_1}(R)^{5/4}+2\bar{c}\gamma^{3/2} {V_1}(R)^{3/2}}{{c_\ell^{2}}\gamma^2 {V_1}(R)^2}\\
	&	\quad +\frac{5\bar{c}}{{c_\ell}\gamma {V_1}(R)}C_2\gamma^{1/2}  {V_1}(R)^{1/2}+\frac{\sqrt{6}\bar{c}+\sqrt{3}\bar{c}}{{c_\ell}\gamma {V_1}(R)}C_3\gamma^{1/2} {V_1}(R)^{1/2}\\
	&\leq \frac{C}{\sqrt{\gamma {V_1}(R)}}
	\end{align*}
	for some suitable constant $C>0$. This completes the proof of the theorem. 
\end{proof}

To end this section, we return to provide the proofs of our technical lemmas.
\begin{proof}[Proof of Lemma \ref{lem:help_integral}]
	Let  $c\geq c_2\sinh^{-1}(1)$ be such that for $\gamma\geq {c_0}$,
	\begin{align*}
		\gamma c_1 {V_1}\Big(\frac{r}{c_2}\Big)-(d-1)r\geq C\gamma {V_1}\Big(\frac{r}{c_2}\Big)
	\end{align*}
	for all $r\geq c$ and some constant $C>0$. Then, for $s\geq c$, {together with the polar integration formula \eqref{eq:polar_integration} we have}
	\begin{align}
		\int_{\H^d\backslash B^d(o,s)}\gamma e^{-\gamma c_1 {V_1}(\frac{{d_1}(x,o)}{c_2})}\; {\mathcal{H}^d_1}(\mathrm{d}x)&=	\omega_{d}\int_{s}^\infty\gamma e^{-\gamma c_1 {V_1}(\frac{r}{c_2})}\sinh^{d-1}(r)\;\mathrm{d}r\nonumber\\
		&\leq \omega_{d}\int_{s}^\infty\gamma e^{-\gamma c_1 {V_1}(\frac{r}{c_2})+(d-1)r}\sinh^{d-1}\Big(\frac{r}{c_2}\Big)\;\mathrm{d}r\nonumber\\
		&\leq \omega_{d}\int_{s}^\infty\gamma e^{-\gamma C {V_1}(\frac{r}{c_2})}\sinh^{d-1}\Big(\frac{r}{c_2}\Big)\;\mathrm{d}r\nonumber\\
		&=\frac{c_2}{C}e^{-\gamma C {V_1}(\frac{s}{c_2})}.\label{eq:first_case}
	\end{align}
For $s\in [0,c]$ we divide the integral into two parts. For the first part we use that there exists a constant ${\overline{c}}>0$ such that $\sinh(x)\leq {\overline{c}}\sinh(\frac{x}{c_2})$ for all $x\leq c$. Then 
\begin{align*}
	\omega_{d}\int_{s}^c\gamma e^{-\gamma c_1 {V_1}(\frac{r}{c_2})}\sinh^{d-1}(r)\;\mathrm{d}r&\leq \omega_{d} {\overline{c}}^{d-1}\int_{s}^c\gamma e^{-\gamma c_1 {V_1}(\frac{r}{c_2})}\sinh^{d-1}\Big(\frac{r}{c_2}\Big)\;\mathrm{d}r\\
	&\leq \frac{c_2}{ c_1}\omega_{d}{\overline{c}}^{d-1}\int_{s}^\infty\frac{ c_1}{c_2}\gamma e^{- c_1\gamma {V_1}(\frac{r}{c_2})}\sinh^{d-1}\Big(\frac{r}{c_2}\Big)\;\mathrm{d}r\\
	&= \frac{c_2}{ c_1}{\overline{c}}^{d-1}  e^{-\gamma  c_1 {V_1}(\frac{s}{c_2})}.
\end{align*}
Similarly to \eqref{eq:first_case} we derive for the second part of the integral
\begin{align*}
	\omega_{d}\int_{c}^\infty\gamma e^{-\gamma c_1 {V_1}(\frac{r}{c_2})}\sinh^{d-1}(r)\;\mathrm{d}r\leq \frac{c_2}{C}e^{-\gamma C {V_1}(\frac{c}{c_2})}\leq \frac{c_2}{C}e^{-\gamma C {V_1}(\frac{s}{c_2})}.
\end{align*}
Altogether, this shows Lemma \ref{lem:help_integral}.
\end{proof}

\begin{proof}[Proof of Lemma \ref{lem:bounded_moments}]
	For a locally finite counting measure $\xi\in\mathcal{N}_{\rm lfc}$ and $z\in B^d(o,R)$  it holds that
	\begin{align}
		\label{eq:dz}
		\lvert D_z\mathcal{L}_{R,\gamma,1}^{(\alpha)}(\xi)\rvert\leq \ell(z,\xi+\delta_z)^{\alpha}+\sum_{x\in\xi}\mathbbm{1}\{\ell(x,\xi)\geq {d_1}(z,x)\}  \ell(x,\xi)^{\alpha}.
	\end{align}
	Jensen's inequality provides $(a+b)^p=2^p(\frac{a}{2}+\frac{b}{2})^p\leq 2^{p-1}(a^p+b^p)$ for $a,b\geq 0$. Hence, together with the Cauchy-Schwarz inequality we have
	\begin{align}
		&\gamma^{\alpha p/d}\E[\lvert D_z\mathcal{L}_{R,\gamma,1}^{(\alpha)}\rvert^p]\nonumber\\&\leq 2^{p-1}\gamma^{\alpha p/d}\E[\ell(z,{\eta}+\delta_{z})^{\alpha p}]+ 2^{p-1}\gamma^{\alpha p/d}\E\Big[\Big(\sum_{x\in\eta_R}\mathbbm{1}\{\ell(x,{\eta})\geq {d_1}(z,x)\}\ell(x,{\eta})^\alpha\Big)^{ p}\Big]\nonumber\\
		&\leq2^{p-1}\gamma^{\alpha p/d}\E[\ell(z,{\eta}+\delta_{z})^{\alpha p}]\nonumber\\
		& \quad +2^{p-1}\gamma^{\alpha p/d}\E\Big[\Big(\sum_{x\in\eta_R}\mathbbm{1}\{\ell(x,{\eta})\geq {d_1}(z,x)\}\Big)^{2p}\Big]^{1/2}\E\Big[\max_{\substack{x\in\eta_R:\\ \ell(x,{\eta})\geq {d_1}(z,x)}}\hspace*{-4mm}\ell(x,{\eta})^{2\alpha p}\Big]^{1/2}\nonumber\\
		&=:2^{p-1}\gamma^{\alpha p/d}\E[\ell(z,{\eta}+\delta_{z})^{\alpha p}]+2^{p-1}\gamma^{\alpha p/d}M_1(z)^{1/2}M_2(z)^{1/2}.\label{eq:bounded_moments}
	\end{align}
	{By Lemma \ref{lem:expectation_ell} (b) the first summand is uniformly bounded in $z$ and $\gamma\geq c_0$.} 

	For $z\in B^d(o,R)$ the factor $M_1(z)$ can be written as a linear combination of terms of the form
	\begin{align*}
		\E\Big[\sum_{(x_1,\dots,x_k)\in \eta_{R, \neq}^k}\mathbbm{1}\{\ell(x_i,{\eta})\geq {d_1}(x_i,z), i=1,\dots,k\}\Big]
	\end{align*}
	for $k\in\{1,\dots,2p\}$. Using the monotonicity relation $\ell(x_i,{\eta}+\sum_{j=1}^k \delta_{x_j})\leq \ell(x_i,{\eta}+\delta_{x_i})$ for $i\in\{1,\dots,k\}$, the multivariate Mecke equation \eqref{eq:Meckemulti} and Hölder's inequality, we have
	\begin{align*}
		&\E\Big[\sum_{(x_1,\dots,x_k)\in\eta_{R,{\neq}}^k}\mathbbm{1}\{\ell(x_i,{\eta})\geq {d_1}(x_i,z), i=1,\dots,k\}\Big]\\
		&=\gamma^k\int_{B^d(o,R)^k}\p\Big(\ell\Big(x_i,\eta_R+\sum_{j=1}^k\delta_{x_i}\Big)\geq {d_1}(x_i,z), i=1,\dots,k\Big)\;{(\mathcal{H}^d_1)^k}(\mathrm{d}(x_1,\dots,x_k))\\
		&\leq \gamma^k\int_{B^d(o,R)^k}\prod_{i=1}^k\p(\ell(x_i,{\eta}+\delta_{x_i})\geq {d_1}(x_i,z))^{1/k}\;{(\mathcal{H}^d_1)^k}(\mathrm{d}(x_1,\dots,x_k))\\
		&\leq \gamma^k\Big(\int_{B^d(o,R)}e^{-\gamma{\mathcal{H}^d_1}(o,\frac{{d_1}(x,z)}{2})/k}\; {\mathcal{H}^d_1}(\mathrm{d}x)\Big)^k\leq \Big(\int_{\mathbb{H}^d}\gamma e^{-\gamma {V_1}(\frac{{d_1}(x,o)}{2})/k}\; {\mathcal{H}^d_1}(\mathrm{d}x)\Big)^k\leq \tilde{c}_1^k
	\end{align*}
	by Lemma \ref{lem:help_integral}. This shows that $M_1$ is uniformly bounded in $z$ and $\gamma\geq c_0$.
	
For the factor $M_2$ we have with Fubini's theorem and the Mecke equation \eqref{eq:Mecke},
\begin{align*}
& \gamma^{2\alpha p/d} \E\bigg[ \max_{\substack{x\in\eta_R:\\ \ell(x,{\eta})\geq {d_1}(z,x)}}\hspace*{-4mm}\ell(x,{\eta})^{2\alpha p} \bigg]  = \gamma^{2\alpha p/d} \int_0^\infty \p\Big(\max_{\substack{x\in\eta_R:\\ \ell(x,{\eta})\geq {d_1}(z,x)}}\hspace*{-4mm}\ell(x,{\eta})^{2\alpha p}\geq u\Big) \;\mathrm{d}u \\
& = \gamma^{2\alpha p/d} \int_0^\infty \p(\exists x\in\eta_R:\ell(x,{\eta})\geq \max\{u^{1/(2\alpha p)}, {d_1}(z,x)\}) \;\mathrm{d}u \\
& \leq \gamma^{2\alpha p/d} \int_0^\infty \E\Big[\sum_{x\in\eta_R}\mathbbm{1}\{\ell(x,{\eta})\geq \max\{u^{1/(2\alpha p)}, {d_1}(z,x)\}\}\Big] \;\mathrm{d}u \\
& = \gamma^{2\alpha p/d} \E\Big[\sum_{x\in\eta_R} \int_0^\infty \mathbbm{1}\{\ell(x,{\eta})\geq \max\{u^{1/(2\alpha p)}, {d_1}(z,x)\}\} \;\mathrm{d}u\Big] \\
& = \gamma^{2\alpha p/d} \E\Big[\sum_{x\in\eta_R} \ell(x,{\eta})^{2\alpha p} \mathbbm{1}\{\ell(x,{\eta})\geq {d_1}(z,x)\} \Big] \\
& = \int_{B^d(o,R)}\gamma^{2\alpha p/d+1} \E[\ell(x,{\eta}+\delta_x)^{2\alpha p} \mathbbm{1}\{\ell(x,{\eta}+\delta_x)\geq {d_1}(z,x)\}] \; {\mathcal{H}^d_1}(\mathrm{d}x).
\end{align*}
Now it follows from Lemma \ref{lem:expectation_ell} (b) and Lemma \ref{lem:help_integral} that
\begin{align*}
\gamma^{2\alpha p/d} \E\bigg[ \max_{\substack{x\in\eta_R:\\ \ell(x,{\eta})\geq {d_1}(z,x)}}\hspace*{-4mm}\ell(x,{\eta})^{2\alpha p} \bigg] & \leq c(\alpha) \int_{B^d(o,R)} \gamma e^{-\gamma {V_1(d_1}(z,x)/2)/2} \; {\mathcal{H}^d_1} (\mathrm{d}x) \\
& \leq c(\alpha) \int_{\H^d} \gamma e^{-\gamma {V_1(d_1}(o,x)/2)/2} \; {\mathcal{H}^d_1}(\mathrm{d}x) \leq c(\alpha) \tilde{c}_1,
\end{align*}
where	$\tilde{c}_1$ does not depend on $\alpha$, $\gamma$ and $z$. Consequently, $\gamma^{2\alpha p/d}M_2(z)$ is uniformly bounded in $\gamma\geq {c_0}$ and $z\in\H^d$. Finally, the uniform bounds for $M_1$ and $\gamma^{2\alpha p/d}M_2$ provide the existence of a constant $C_{1,p}>0$ for which $\gamma^{\alpha p/d}\E[\lvert D_z\mathcal{L}_{R,\gamma,1}^{(\alpha)}\rvert^p]\leq C_{1,p}$ for all $z\in\H^d$ and $\gamma\geq {c_0}$.
	
	For the second-order difference operator, using Jensen's inequality as above, we have
	\begin{align}
		\label{eq:bounded moments_second_order}
		&\gamma^{\alpha p/d}\E[\lvert D_{z_1,z_2}^2\mathcal{L}_{R,\gamma,1}^{(\alpha)}\rvert^p]\nonumber\\&\leq 2^{p-1}\gamma^{\alpha p/d}\E[\lvert D_{z_1}\mathcal{L}_{R,\gamma,1}^{(\alpha)}(\eta_R)\rvert^p]+2^{p-1}\gamma^{\alpha p/d}\E[\lvert D_{z_1}\mathcal{L}_{R,\gamma,1}^{(\alpha)}(\eta_R+\delta_{z_2})\rvert^p]
	\end{align}
	for $z_1,z_2\in B^d(o,R)$.
	The first summand is uniformly bounded since we have already shown that the left-most term in \eqref{eq:bounded_moments} is uniformly bounded. For the second expression we can use similar arguments as above. Analogously to \eqref{eq:dz} we have
	\begin{align*}
		\lvert D_{z_1}\mathcal{L}_{R,\gamma,1}^{(\alpha)}(\eta_R+\delta_{z_2})\rvert&\leq \ell(z_1,{\eta}+\delta_{z_2}+\delta_{z_1})^\alpha\\&\quad+\sum_{x\in\eta_R+\delta_{z_2}}\mathbbm{1}\{\ell(x,{\eta}+\delta_{z_2})\geq {d_1}(z_1,x)\}\max_{\substack{x\in\eta_R+\delta_{z_2}:\\ \ell(x,{\eta}+\delta_{z_2})\geq {d_1}(z_1,x)}}\hspace*{-4mm}\ell(x,{\eta}+\delta_{z_2})^\alpha.
	\end{align*}
	Due to monotonicity one has
	\begin{align*}
		\ell(z_1,{\eta}+\delta_{z_2}+\delta_{z_1})^\alpha&\leq \ell(z_1,{\eta}+\delta_{z_1})^\alpha,\\
		\sum_{x\in\eta_R+\delta_{z_2}}\mathbbm{1}\{\ell(x,{\eta}+\delta_{z_2})\geq {d_1}(z_1,x)\}&\leq 1+\sum_{x\in\eta_R}\mathbbm{1}\{\ell(x,{\eta})\geq {d_1}(z_1,x)\}
	\end{align*}
	and
	\begin{align*}
		\max_{\substack{x\in\eta_R+\delta_{z_2}:\\ \ell(x,{\eta}+\delta_{z_2})\geq {d_1}(z_1,x)}}\hspace*{-4mm}\ell(x,{\eta}+\delta_{z_2})^\alpha\leq \ell(z_2,{\eta}+\delta_{z_2})^\alpha+\max_{\substack{x\in\eta_R:\\ \ell(x,{\eta})\geq {d_1}(z_1,x)}}\hspace*{-4mm}\ell(x,{\eta})^\alpha,
	\end{align*}
	which yields that the second summand in \eqref{eq:bounded moments_second_order} is also uniformly bounded for $z_1,z_2\in \mathbb{H}^d$ and $\gamma\geq {c_0}$.
\end{proof}

\begin{proof}[Proof of Lemma \ref{lem:prop_seconddiff_neq_0}]
	We have 
	\begin{align*}
	\nonumber	D_{z_1,z_2}^2\mathcal{L}_{R,\gamma,1}^{(\alpha)}=\sum_{x\in\eta_R}D_{z_1,z_2}^2\ell(x,{\eta})^\alpha+D_{z_1}\ell(z_2,{\eta}+\delta_{z_2})^\alpha+D_{z_2}\ell(z_1,{\eta}+\delta_{z_1})^\alpha.
	\end{align*}
	This second-order difference operator can only be non-zero if at least one of the three summands above is non-zero.
	Hence,
	\begin{align}
	\nonumber	\p(	D_{z_1,z_2}^2\mathcal{L}_{R,\gamma,1}^{(\alpha)}\neq 0)&\leq \p(\exists x\in\eta_R:D_{z_1,z_2}^2\ell(x,{\eta})^\alpha\neq 0)+\p(D_{z_1}\ell(z_2,{\eta}+\delta_{z_2})^\alpha\neq 0)\\
		&\qquad+\p(D_{z_2}\ell(z_1,{\eta}+\delta_{z_1})^\alpha\neq 0).\label{eq:25-07-24}
	\end{align}
	Since $D_{z_2}\ell(z_1,{\eta}+\delta_{z_1})^\alpha\neq 0$ requires $\ell(z_1,{\eta}+\delta_{z_1})\geq {d_1}(z_1,z_2)$ {and ${d_1}(z_1,z_2)\leq {d_1}(o,z_1)$,} we get
	\begin{align*}
		&\p(D_{z_2}\ell(z_1,{\eta}+\delta_{z_1})^\alpha\neq 0)\\
		&\leq \p(\ell(z_1,{\eta}+\delta_{z_1})\geq {d_1}(z_1,z_2))\,\mathbbm{1}\{{d_1}(z_1,z_2)\leq {d_1}(o,z_1)\}\\
		&=\p(\eta_R(B^d(z_1,{d_1}(z_1,z_2))\cap B^d(o,{d_1}(z_1,o)))=0)\,\mathbbm{1}\{{d_1}(z_1,z_2)\leq {d_1}(o,z_1)\}\\
		&\leq {e^{-\gamma {V_1(d_1}(z_1,z_2)/2)}},
	\end{align*}
where we used \eqref{eqn:intersection_balls} in the last step, and, analogously,
$$
\p(D_{z_1}\ell(z_2,{\eta}+\delta_{z_2})^\alpha\neq 0)\leq e^{-\gamma {V_1(d_1}(z_1,z_2)/2)}.
$$
For the first summand in \eqref{eq:25-07-24} we use the Mecke equation \eqref{eq:Mecke} and the observation that $D_{z_1,z_2}^2\ell(x,{\eta})^\alpha\neq 0$ requires $\ell(x,{\eta})\geq\max\{{d_1}(z_1,x),{d_1}(z_2,x)\}$ {and $\max\{d_1(z_1,x),d_1(z_2,x)\}\leq d_1(o,x)$} for $x\in B^d(0, R)$ and get with \eqref{eqn:intersection_balls} and Lemma \ref{lem:help_integral},
	\begin{align*}
		\p(\exists x\in\eta_R:D_{z_1,z_2}^2\ell(x,{\eta})^\alpha\neq 0)
		&\leq\E\Big[\sum_{x\in\eta_R}\mathbbm{1}\{D_{z_1,z_2}^2\ell(x,{\eta})^\alpha\neq 0\}\Big]\\
		&=\int_{B^d(o,R)}\gamma\p(D_{z_1,z_2}^2\ell(x,\eta+\delta_x)^\alpha\neq 0)\; {\mathcal{H}^d_1}(\mathrm{d}x)\\
		&\leq\int_{B^d(o,R)}\gamma {\id\{ \max\{d_1(z_1,x),d_1(z_2,x)\}\leq d_1(o,x) \}} \\
		&  \qquad \times \p(\ell(x,{\eta}+\delta_x)\geq\max\{{d_1}(z_1,x),{d_1}(z_2,x)\}) \; {\mathcal{H}^d_1}(\mathrm{d}x)\\
		&\leq \int_{B^d(o,R)}\gamma e^{-\gamma {V_1(\max\{d_1(z_1,x),d_1(z_2,x)\}/2)}}\; {\mathcal{H}^d_1}(\mathrm{d}x)\\
		&\leq \int_{\mathbb{H}^d\backslash B^d(z_1,\frac{{d_1}(z_1,z_2)}{2})}\gamma e^{-\gamma {V_1(d_1(z_1,x)/2)}}\; {\mathcal{H}_1^d}(\mathrm{d}x)\\
		& \quad +\int_{\mathbb{H}^d\backslash B^d(z_2,\frac{{d_1}(z_1,z_2)}{2})}\gamma e^{-\gamma {V_1(d_1(z_2,x)/2)}}\; {\mathcal{H}^d_1}(\mathrm{d}x)\\
		&\leq2 \tilde{c}_1e^{-\tilde{c}_2\gamma {V_1(d_1(z_1,z_2)/4)}},
	\end{align*}
 which provides the result.
\end{proof}

We prepare the proof of Lemma \ref{lem:D_lower_bound} with the following geometric result.

\begin{lemma}\label{lemma:lower_var_bound_epsilon_balls}
There exists a constant $\overline{\varepsilon}>0$ only depending on $d$ such that for all $z\in B^d(o,3)^c$ and all $x\in B^d(z,1)\backslash\{z\}$ the set 
$$
B^d(z,d_1(x,z))^c\cap B^d(x,d_1(x,z)) \cap B^d(o, d_1(x,o))
$$
contains a ball of radius $\overline{\varepsilon} d_1(x,z)$.
\end{lemma}

\begin{proof}
For $z\in B^d(o,3)^c$ and $x\in B^d(z,1)\backslash\{z\}$ let $\tau:[0,1]\to\mathbb{H}^d$ be the geodesic segment from $o$ to $x$, i.e.\ $\tau(0)=o$ and $\tau(1)=x$. Then,
$$
B^d(\tau(u), d_1(x,\tau(u)))\subset B^d(\tau(v), d_1(x,\tau(v)))
$$
for $u>v$, i.e.\ shifting the origin in the direction of $x$ along a geodesic ray decreases the considered set. Since $d_1(x,z)\leq 1$, $d_1(o,z)\geq 3$ and $d_1(\tau(u),z)$ is continuous in $u$, there exists $t\in[0,1]$ satisfying $d_1(\tau(t),z)=3$. Hence, to show the lemma, it is enough to only consider $z\in B^d(o,3)^c$ with $d_1(z,o)=3$ and by rotational invariance, it is sufficient to show the statement for a fixed $\hat{z}$ with $d_1(\hat{z},o)=3$.

Denote by $\overline{B}^d(y,r)$ the closed ball with centre $y$ and radius $r$. For $x\in \overline{B}^d(\hat{z},1)$ let {$R(x)$} be the largest radius of an open ball contained in the set {$B^d(\hat{z},d_1(x,\hat{z}))^c\cap B^d(x,d_1(x,\hat{z})) \cap B^d(o, d_1(x,o))$}. Note that $R$ is continuous on $\overline{B}^d(\hat{z},1)\setminus\{\hat{z}\}$. Assume that the function $h: \overline{B}^d(\hat{z},1)\setminus\{\hat{z}\}\to\R, x \mapsto \frac{R(x)}{d_1(x,\hat{z})}$ is not bounded away from $0$. Then there exists a sequence {$(x_n)_{n\in\mathbb{N}}$ in $\overline{B}^d(\hat{z},1)\setminus\{\hat{z}\}$ such that
$$
\lim_{n\to\infty} \frac{R(x_n)}{d_1(x_n,\hat{z})} = 0.
$$ 
Since $\overline{B}^d(\hat{z},1)\setminus B^d(\hat{z},u)$ is compact for any $u\in(0,1)$ and $h\neq 0$ on $\overline{B}^d(\hat{z},1)\setminus\{\hat{z}\}$, $h$ is bounded away from zero on $\overline{B}^d(\hat{z},1)\setminus B^d(\hat{z},u)$. This implies $x_n\to \hat{z}$ as $n\to\infty$}. Define $r_n:={d_1(x_n,\hat{z})}$ for $n\in\mathbb{N}$. We have
\begin{align*}
& \liminf_{n\to\infty} \frac{1}{r_n^d} \mathcal{H}^d_1(B^d({\hat{z}},r_n)^c\cap B^d({x_n},r_n) \cap B^d({o}, d_1(x_n,{o}))) \\
& \geq \liminf_{n\to\infty} \frac{1}{r_n^d} \big( \mathcal{H}^d_1(B^d({x_n},r_n) \cap B^d({o}, d_1(x_n,{o}))) - \mathcal{H}^d_1(B^d(x_n,r_n)\cap B^d({\hat{z}},r_n)) \big) \\
& \geq \liminf_{n\to\infty} \frac{1}{r_n^d} \big( \mathcal{H}^d_1(B^d({x_n},r_n) \cap B^d({y_n}, d_1(x_n,{y_n}))) - \mathcal{H}^d_1(B^d(x_n,r_n)\cap B^d({\hat{z}},r_n)) \big),
\end{align*}
where {$y_n$} is the point on the geodesic between $x_n$ and ${o}$ such that $d_1(x_n,{y_n})=2r_n$. Since hyperbolic balls of small radii can be approximated by Euclidean balls in appropriate tangent spaces, we obtain
\begin{align*}
& \liminf_{n\to\infty} \frac{1}{r_n^d} \mathcal{H}^d_1(B^d({\hat{z}},r_n)^c\cap B^d({x_n},r_n) \cap B^d({o}, d_1(x_n,{o}))) \\
& \geq V_d^{\rm {E}}(B^d_{\rm {E}}(0,1)\cap B^d_{\rm {E}}(v_2,2) ) - V_d^{\rm {E}}(B^d_{\rm {E}}(0,1)\cap B^d_{\rm {E}}(v_1,1) ) =: c_1
\end{align*}
with $v_1,v_2\in\mathbb{R}^d$ such that $|v_1|_{\rm {E}}=1$ and $|v_2|_{\rm {E}}=2$ and $V_d^{\rm E}$ denoting the $d$-dimensional Euclidean Lebesgue measure. Obviously, we have $c_1>0$. Next fix $\tilde{\varepsilon}\in(0,1)$. Then, one has
\begin{align*}
& \limsup_{n\to\infty} \frac{1}{r_n^d} \mathcal{H}_1^d(\{w\in B^d({\hat{z}},r_n)^c\cap B^d(x_n,r_n)\cap B^d({o}, d_1(x_n,{o})): \\
& \hspace{2.5cm} B^d(w,\tilde{\varepsilon}r_n)\not\subseteq B^d({\hat{z}},r_n)^c\cap B^d(x_n,r_n)\cap B^d({o}, d_1(x_n,{o}))\}) \\
& \leq \limsup_{n\to\infty} \frac{1}{r_n^d} \Big( \mathcal{H}_1^d(B^d({\hat{z}},(1+\tilde{\varepsilon})r_n)\setminus B^d({\hat{z}},r_n)) + \mathcal{H}^d_1(B^d(x_n,r_n)\setminus B^d({x_n},(1-\tilde{\varepsilon})r_n)) \\
& \hspace{2.25cm} + \mathcal{H}_1^d((B^d({o},d_1(x_n,{o}))\setminus B^d({o},d_1(x_n,{o})-\tilde{\varepsilon}r_n))\cap B^d(x_n,r_n) \Big) \\
& \leq \limsup_{n\to\infty} \frac{1}{r_n^d} \Big( \mathcal{H}_1^d(B^d(o,(1+\tilde{\varepsilon})r_n)\setminus B^d(o,r_n)) + \mathcal{H}^d_1(B^d({o},r_n)\setminus B^d({o},(1-\tilde{\varepsilon})r_n)) \\
& \hspace{2.25cm} + \mathcal{H}_1^d(B^d({o},d_1(x_n,{o}))\cap B^d(x_n,r_n))- \mathcal{H}_1^d(B^d({o},d_1(x_n,{o})-\tilde{\varepsilon}r_n)\cap B^d(x_n,r_n)) \\
& = \kappa_d \Big( (1+\tilde{\varepsilon})^d-1 + 1 - (1-\tilde{\varepsilon})^d + {1\over 2} - v(\tilde{\varepsilon}) \Big) =:g(\tilde{\varepsilon}),
\end{align*}
where $v(\tilde{\varepsilon}):=\kappa_d^{-1}V_d^{\rm E}(\{(x_1,\ldots,x_d)\in B_{\rm E}^d(0,1):x_d\geq\tilde{\varepsilon}\})$ and where we used again convergence of hyperbolic balls with small radii to Euclidean ones. Combining the previous estimates, we see that
\begin{align*}
& \liminf_{n\to\infty} \frac{1}{r_n^d} \mathcal{H}_1^d(\{w\in B^d({\hat{z}},r_n)^c\cap B^d(x_n,r_n)\cap B^d({o}, d_1(x_n,{o})): \\
& \hspace{2.5cm} B^d(w,\tilde{\varepsilon}r_n)\subseteq B^d({\hat{z}},r_n)^c\cap B^d(x_n,r_n)\cap B^d({o}, d_1(x_n,{o}))\}) \\
& \geq c_1- g(\tilde{\varepsilon}).
\end{align*}
Since $g(t)\to 0$ as $t\to 0$, the right-hand side is positive for $\tilde{\varepsilon}$ sufficiently small. Thus, for $n$ large enough there exists $w\in B^d({\hat{z}},r_n)^c\cap B^d(x_n,r_n)\cap B^d({o}, d_1(x_n,{o}))$ such that $B^d(w,\tilde{\varepsilon}r_n)\subseteq B^d({\hat{z}},r_n)^c\cap B^d(x_n,r_n)\cap B^d({o}, d_1(x_n,{o}))$ and, thus, ${h(x_n)\geq \tilde{\varepsilon}}$. This is a contradiction so that we have shown that $h$ must be bounded away from zero, which proves the desired statement.
\end{proof}

\begin{proof}[Proof of Lemma \ref{lem:D_lower_bound}]
Fix $c > 0$ such that $c_\gamma:=\gamma^{-\alpha/d}c\leq 1$ for all $\gamma\geq c_0$ and let {$z\in\H^d$} be such that ${d_1(o,z) \geq 3}$. Denote by $A_\gamma := \{x \in \H^d \,:\, c_\gamma^{1/\alpha} \leq d_1(x,z) \leq 2c_\gamma^{1/\alpha}  \}$ the closed annulus with radii $c_\gamma^{1/\alpha}  < 2c_\gamma^{1/\alpha} $ around $z$.
For all $x\in A_\gamma$ with $d_1(x,o)\geq d_1(z,o)$ denote by $x^*\in A_\gamma$ the point on the geodesic from $x$ to $z$ satisfying $d_1(x^*,z)=c_\gamma^{1/\alpha}$. Then, by the convexity of the hyperbolic distance function along geodesics (see e.g.\ \cite[Theorem 1.4.2]{AVS}) one has $d_1(x^*,o)\leq d_1(x,o)$ and hence,
$$
B^d(o,{d_1}(o,x^*)) \cap B^d(x^*, {d_1}(z,x^*))\cap A_\gamma \subseteq B^d(o,{d_1}(o,x)) \cap B^d(x, {d_1}(z,x))\cap A_\gamma.
$$
The left-hand side can be rewritten as
$$
B^d(o,{d_1}(o,x^*)) \cap B^d(x^*, {d_1}(z,x^*)) \cap A_\gamma=B^d(z,d_1(x^*,z))^c\cap B^d(x,d_1(x^*,z)) \cap B^d(o, d_1(x^*,o))
$$
so that, by Lemma \ref{lemma:lower_var_bound_epsilon_balls}, $B^d(o,{d_1}(o,x^*)) \cap B^d(x^*, d_1(z,x)) \cap A_\gamma$ contains a ball of radius $\overline{\varepsilon} d_1(x^*,z) = \overline{\varepsilon} c_\gamma^{1/\alpha} = \overline{\varepsilon} c^{1/\alpha} \gamma^{-1/d}$. Hence, we can choose $\eps>0$ such that the intersection $B^d(o,{d_1}(o,x)) \cap B^d(x, {d_1}(z,x)) \cap A_\gamma$ contains a ball of radius $\eps_\gamma:=\gamma^{-1/d}\eps$ for all $\gamma\geq c_0$ and $x\in A_\gamma$. Note that $\varepsilon$ does not depend on the choice of $z\in B^d(o,3)^c$.

Now consider the event
\begin{equation*}
E_{z,\gamma} := \{ \text{$\eta \cap A_\gamma$ is an $\eps_\gamma$-dense subset of $A_\gamma$ and $\eta \cap B^d(z,c_\gamma^{1/\alpha} ) = \emptyset$} \},
\end{equation*}
where a set is called $\eps_\gamma$-dense in $A_\gamma$ if for all $y\in A_\gamma$ there exists a point of the $\eps_{\gamma}$-dense set having distance smaller than or equal to $\eps_{\gamma}$ to $y$. We claim that, conditionally on $E_{z,\gamma}$, one has $D_z \mathcal{L}_{R,\gamma,1}^{(\alpha)} \geq c_\gamma$. Let us first show that if $E_{z,\gamma}$ occurs, $z$ is not the radial nearest neighbour of any $x \in \eta _R$. Indeed, first let $x \in \eta \cap A_\gamma$ be such that $d_1(o,x)\geq d_1(o,z)$ (so that $z$ could potentially be the radial nearest neighbour of $x$). Then, by the choice of $\eps_\gamma$ and since $\eta \cap A_\gamma$ is $\eps_\gamma$-dense in $A_\gamma$, there exists $w \in \eta \cap A_\gamma$ which lies in $B^d(o,d_1(o,x)) \cap B^d(x, d_1(z,x))$, which in particular shows that $z$ is not the radial nearest neighbour of $x$. Now let $x \in \eta \setminus B^d(z,2c_\gamma^{1/\alpha} )$ with $d_1(o,x)\geq d_1(o,z)$. 
 Let $x_B$ be the intersection point of the geodesic segment $[z,x]$ with the sphere $\partial B^d(z,2c_\gamma^{1/\alpha})$, and let again $w \in \eta \cap A_\gamma$ be such that $w \in B^d(o, d_1(o,x_B)) \cap B^d(x_B, d_1(z,x_B))$. Then, using again the convexity of the hyperbolic distance function along geodesics one has
\begin{equation*}
d_1(o,w) \leq d_1(o,x_B) \leq \max \{d_1(o,x), d_1(o,z)  \} = d_1(o,x).
\end{equation*}
Additionally,  using the fact that $x_B$ lies on the geodesic segment $[z,x]$ we see that
\begin{equation*}
d_1(w,x) \leq d_1(w,x_B) + d_1(x_B,x) \leq d_1(z,x_B) + d_1(x_B,x) = d_1(z,x).
\end{equation*}
As above, these two conditions combined show that $z$ is not the radial nearest neighbour of $x$. 
This proves that, assuming $E_{z,\gamma}$, adding $z$ does not delete any edges from the radial spanning tree and hence adds exactly one edge between $z$ and its radial nearest neighbour. Since in the case of $E_{z,\gamma}$ there are no points of $\eta$ in $B^d(z,c_\gamma^{1/\alpha})$, the length of this edge is at least $c_\gamma^{1/\alpha}$. In other words, we have proven that
\begin{equation*}
\PP(D_z \mathcal{L}_{R,\gamma,1}^{(\alpha)} \geq c_\gamma) \geq \PP(E_{z,\gamma}).
\end{equation*}
To estimate the latter probability fix a maximal collection $\{B_i\}_{i\in I}$ of disjoint $\eps_\gamma/3$-balls in $A_\gamma$. Then any choice of one point from each ball is clearly $\eps_\gamma$-dense in $A_\gamma$. Let $c_u>0$ be such that $\sinh(x)\leq c_ux$ for all $x\in [0,2]$ and recall that $\sinh(x)\geq x$ for all $x\geq 0$ as well as the formula \eqref{eq:volume_ball} for the volume of hyperbolic balls. Then, the number of balls $\{B_i\}_{i\in I}$ is obviously no more than 
\begin{align*}
	k_\gamma := \frac{\cH^d_1(B^d(o,2c_\gamma^{1/\alpha}))}{\cH^d_1(B^d(o,\eps_{\gamma}/3))}\leq \frac{c_u^{d-1}2^dc_\gamma^{d/\alpha}}{(\eps_{\gamma}/3)^d}=\frac{c_u^{d-1}3^{d}2^dc^{d/\alpha}}{\eps^d}=:k,
\end{align*}
which is independent from $\gamma$. This shows that 
\begin{align*}
\PP(E_{z,\gamma}) &\geq \PP(\eta(B_i) = 1 \, \forall i\in I, \, \eta(B^d(z,c_\gamma^{1/\alpha} ))=0) \\
&\geq \left(\gamma \cH^d_1(B^d(o,\eps_\gamma/3))e^{-\gamma \cH^d_1(B^d(o,\eps_\gamma/3) ) }\right)^k \, e^{-\gamma \cH^d_1(B^d(o,c_\gamma^{1/\alpha} ))}\\
&\geq \left( \frac{\kappa_d\eps^d}{3^d} e^{- c_u^{d-1}\kappa_d\eps^d/3^d }\right)^k \, e^{- c_u^{d-1}\kappa_dc^{d/\alpha}}=: C >0
\end{align*}
and completes the proof.
\end{proof}

\subsection*{Acknowledgement}
DR and CT were supported by the German Research Foundation (DFG) via CRC/TRR 191 \textit{Symplectic Structures in Geometry, Algebra and Dynamics}. MS and CT were also supported by the DFG priority program SPP 2265 \textit{Random Geometric Systems}.

\end{document}